\documentclass{amsart}
\pdfoutput=1

\usepackage{amssymb}
\usepackage{microtype}
\usepackage{tikz}
\usepackage{booktabs}
\usepackage[pdftex,colorlinks,citecolor=black,linkcolor=black,urlcolor=black,bookmarks=false]{hyperref}
\def\MR#1{\href{http://www.ams.org/mathscinet-getitem?mr=#1}{MR#1}}
\def\arXiv#1{arXiv:\href{http://arXiv.org/abs/#1}{#1}}
\usepackage{doi}

\newtheorem{theorem}{Theorem}[section]
\newtheorem{proposition}[theorem]{Proposition}
\newtheorem{lemma}[theorem]{Lemma}
\newtheorem{corollary}[theorem]{Corollary}
\newtheorem{conjecture}[theorem]{Conjecture}

\numberwithin{equation}{section}
\numberwithin{table}{section}

\newcommand{\ZZ}{\mathbb{Z}}
\newcommand{\RR}{\mathbb{R}}
\newcommand{\CC}{\mathbb{C}}
\newcommand{\EE}{\mathbb{E}}
\newcommand{\HH}{\mathbb{H}}

\newcommand{\SL}{\mathop{\mathrm{SL}}}
\newcommand{\erf}{\mathop{\mathrm{erf}}}

\makeatletter
\newsavebox\tboxa
\newsavebox\tboxb
\newlength\tdima
\newcommand*{\oversymb}{\mathpalette\@oversymb}
\newcommand*{\@oversymb}[2]{%
    \sbox{\tboxa}{$\m@th#1\mathrm{#2}$}%
    \setbox\tboxb\null%
    \ht\tboxb\ht\tboxa%
    \dp\tboxb\dp\tboxa%
    \wd\tboxb\wd\tboxa%
    \sbox{\tboxa}{$\m@th#1{#2}$}%
    \setlength\tdima{\the\wd\tboxa}%
    \addtolength\tdima{-\the\wd\tboxb}%
    \sbox{\tboxb}{$\m@th#1\hskip\tdima\overline{\xusebox{\tboxb}}$}%
    \rlap{\usebox\tboxb}{\usebox\tboxa}}
\newcommand*{\xusebox}[1]{\mathord{{\usebox{#1}}}}

\newcommand{\abs}[1]{\left\lvert #1 \right\rvert}
\newcommand{\ang}[1]{\left\langle #1 \right\rangle}
\newcommand{\paren}[1]{\left( #1 \right)}
\newcommand{\set}[1]{\left\{ #1 \right\}}
\newcommand{\ol}[1]{\overline{#1}}
\newcommand{\os}[1]{\oversymb{#1}}
\newcommand{\vol}{\mathop{\mathrm{vol}}}
\newcommand{\wht}[1]{\widehat{#1}}

\title{Sphere packing bounds via spherical codes}

\author{Henry Cohn}
\address{Microsoft Research New England\\
One Memorial Drive\\
Cambridge, MA 02142} \email{cohn@microsoft.com}

\author{Yufei Zhao}
\address{Department of Mathematics\\
Massachusetts Institute of Technology\\
Cambridge, MA 02139} \email{yufeiz@mit.edu}

\date{December 20, 2013}

\thanks{Zhao was supported by an internship at Microsoft Research
New England.}

\begin{document}

\begin{abstract}
The sphere packing problem asks for the greatest density of a packing
of congruent balls in Euclidean space. The current best upper bound in
all sufficiently high dimensions is due to Kabatiansky and Levenshtein
in 1978. We revisit their argument and improve their bound by a
constant factor using a simple geometric argument, and we extend the
argument to packings in hyperbolic space, for which it gives an
exponential improvement over the previously known bounds. Additionally,
we show that the Cohn-Elkies linear programming bound is always at
least as strong as the Kabatiansky-Levenshtein bound; this result is
analogous to Rodemich's theorem in coding theory.  Finally, we develop
hyperbolic linear programming bounds and prove the analogue of
Rodemich's theorem there as well.
\end{abstract}

\maketitle

\section{Introduction} \label{sec:intro}

What is the densest arrangement of non-overlapping, congruent balls in
$\RR^n$?  This problem has a long history and has been extensively
studied \cite{CS99}, and it has strong connections with physics and
information theory \cite{C10}. With the proof of Kepler's conjecture by
Hales \cite{Hal05}, the sphere packing problem has been solved in up to three
dimensions, but no proof of optimality is known in any higher
dimension, and there are only a few dozen cases in which there are even
plausible conjectures for the densest packing. In $\RR^8$ and
$\RR^{24}$ there are upper bounds that are remarkably close to the
densities of the $E_8$ and Leech lattices, respectively; for example,
Cohn and Kumar \cite{CK04,CK09} came within a factor of $1+10^{-14}$ of the
density of $E_8$ and a factor of $1 + 1.65 \cdot 10^{-30}$ of the
density of the Leech lattice.  However, in most dimensions we must be
content with much cruder bounds. In this paper, we will slightly
improve the best upper bounds known in high dimensions, show how to obtain
them via linear programming bounds, and extend them to hyperbolic
space.

The \emph{density} of a sphere packing in $\RR^n$ is the fraction of
space covered by the balls in the packing.  More precisely, let
$B_R^n(x)$ denote the ball of radius $R$ centered at $x$; then the
density of a packing is the limit as $R \to \infty$ of the fraction of
$B_R^n(x)$ covered by the packing (the limit is independent of $x$ if it
exists). Of course this limit need not exist, but one can replace it
with the \emph{upper density} defined with a limit superior, and one
can show that the least upper bound of the upper densities of
all sphere packings in $\RR^n$ is actually achieved as the density of a
packing (see \cite{G63}). Let $\Delta_{\RR^n}$ denote this maximal packing density.

A \emph{spherical code} in dimension $n$ with minimum angle $\theta$ is
a set of points on the unit sphere in $\RR^n$ with the property that no
two points subtend an angle less than $\theta$ at the origin. In other
words, $\ang{x,y} \leq \cos \theta$ for all pairs of distinct points
$x, y$ in the spherical code. Let $A(n,\theta)$ denote the greatest
size of such a spherical code.

In this paper, we consider the problem of finding upper bounds for
packing density. Linear programming bounds have proven to be a powerful
tool. This technique was first developed by Delsarte \cite{Del72} in
the setting of error-correcting codes, and his method can be extended
to many other settings. In particular, Delsarte, Goethals, and Seidel
\cite{DGS77} and Kabatiansky and Levenshtein \cite{KL78} independently
formulated a linear program for proving upper bounds on $A(n,\theta)$.
Using this approach, Kabatiansky and Levenshtein found excellent upper
bounds on $A(n,\theta)$ for large $n$, and they then applied a
geometric argument to deduce a bound on $\Delta_{\RR^n}$. Their upper
bound is currently the best bound known for $n \geq 115$ (see
Appendix~\ref{app:numerical}). It has the asymptotic form
\begin{equation}
  \label{eq:KL}
  \Delta_{\RR^n} \leq 2^{-(0.5990\ldots+o(1))n},
\end{equation}
while the best lower bound known remains $2^{-(1+o(1))n}$ despite
recent improvements \cite{V11,V13}.

Cohn and Elkies \cite{CE03} found a more direct approach to bounding
sphere packing densities, with no need to consider spherical codes.
Their technique set new records in every case with $n \ge 4$ for which
the calculations were carried out; see Appendix~\ref{app:numerical} for more
details, and see
Theorem~1.4 in \cite{LOV12} for subsequent improvements when $n = 4$, $5$, $6$, $7$,
and $9$. However, despite the evidence from low dimensions, the
asymptotic behavior of the Cohn-Elkies bound is far from obvious and it
has been unclear whether it improves on, or even matches, the
Kabatiansky-Levenshtein bound asymptotically.  Until this paper, it was
only known how to use the Cohn-Elkies linear program to match the
``second-best bound'' by Levenshtein~\cite{Lev79} (see Section~6 of
\cite{CE03}).

The purpose of this paper is fourfold. In Section~\ref{sec:geometric}
we improve the Kabatiansky-Levenshtein bound by a constant factor by
giving a simple modification of their geometric argument relating
spherical codes to sphere packings. (This does not change the
exponential decay rate in bound \eqref{eq:KL}). In Section~\ref{sec:LP}
we show that in every dimension $n$, the Cohn-Elkies linear program can
always match the Kabatiansky-Levenshtein approach. This further
demonstrates the power of the linear programming bound for sphere
packing.  In Section~\ref{sec:hyperbolic} we prove an analogue of the
Kabatiansky-Levenshtein bound in hyperbolic space.  The resulting bound
behaves the same as \eqref{eq:KL} asymptotically, and it is
exponentially better than the best bound previously known in hyperbolic
space.  Finally, in Section~\ref{sec:hypLP}, we develop the theory of
hyperbolic linear programming bounds (based partly on unpublished work
of Cohn, Lurie, and Sarnak) and prove that they too subsume the
Kabatiansky-Levenshtein approach.

\section{Geometric argument} \label{sec:geometric}

In all sufficiently high dimensions, the best upper bound currently
known for sphere packing density is given by Kabatiansky and
Levenshtein \cite{KL78} (see also Chapter~9 of \cite{CS99} and
Chapter~8 of \cite{Zon99}). They first obtain an upper bound on $A(n,
\theta)$ using linear programming and then use the inequality
\begin{equation}
\label{eq:lift}
\Delta_{\RR^n} \leq \sin^n (\theta/2) A(n+1, \theta).
\end{equation}
The inequality was derived using a simple geometric argument. Here we
improve it using an equally simple argument.

\begin{proposition} \label{prop:compare}
For all $n \ge 1$ and $\pi/3 \le \theta \le \pi$,
\begin{equation}
\label{eq:prop-compare}
\Delta_{\RR^n} \leq \sin^n (\theta/2) A(n, \theta).
\end{equation}
\end{proposition}

Since the unit sphere in $\RR^n$ can be embedded in the unit sphere in
$\RR^{n+1}$ via a hyperplane through the origin, we always have
$A(n,\theta) \le A(n+1,\theta)$, with strict inequality when $\theta
\le \pi/2$.  The applications of \eqref{eq:lift} have $\pi/3 \le
\theta \le \pi/2$, so Proposition~\ref{prop:compare} will be a
strict (though small) improvement.  Neither inequality is useful in low
dimensions; for example, when $n=2$ and $\theta=\pi/3$,
Proposition~\ref{prop:compare} says that $\Delta_{\RR^2} \le 3/2$.
However, these inequalities are valuable in high dimensions.

For the sake of comparison, let us first recall the proof of \eqref{eq:lift}.

\begin{proof}[Proof of \eqref{eq:lift}]
Suppose we have a sphere packing in $\RR^n$ of density $\Delta$ using
unit spheres. Consider a sphere $S_R^{n}$ in $\RR^{n+1}$ of radius $R$
(to be chosen later), and place the sphere packing in $\RR^n$ onto a
hyperplane through the center of $S_R^{n}$, with the packing translated
so that at least $\Delta R^n$ of the sphere centers are contained in
$S_R^{n}$. This is always possible by an averaging argument: a randomly
chosen translation will lead to an average of $\Delta R^n$ sphere centers
in $S_R^{n}$.  Project
the sphere centers onto the upper hemisphere of $S_R^{n}$, orthogonally
to the hyperplane. The projections of the sphere centers are still at
least distance two apart, and thus separated by angles of at least
$\theta$, where $\sin(\theta/2) = 1/R$. Therefore, $\Delta R^n \leq
A(n+1,\theta)$, which is the bound that we wanted to prove, and we can
achieve any angle by choosing $R$ accordingly.
\end{proof}

Our motivation for revisiting this argument is that it feels somewhat
unnatural to lift to a higher dimension in the process. Our proposition
shows that a stronger inequality can be obtained without going to a
higher dimension.  The proof is similar to the techniques of
\cite{HST10} and \cite{BM07}, but this application appears to be new.

\begin{figure}
\centering
\begin{tikzpicture}[scale=1]
\path[use as bounding box] (-0.25,-0.25) rectangle (3.75,2.848076);
\draw[->] (1.8,1.35) -- (1.273142,0.500046);
\draw[->] (1.8,1.35) -- (2.622544,0.781298);
\draw[->] (1.8,1.35) -- (0.897620,1.780942);
\draw[->] (1.8,1.35) -- (2.263786,2.235947);
\draw[fill=black] (0,0) circle (1.5pt);
\draw[fill=black] (1,0) circle (1.5pt);
\draw[fill=black] (2,0) circle (1.5pt);
\draw[fill=black] (3,0) circle (1.5pt);
\draw[fill=black] (0.5,0.866025) circle (1.5pt);
\draw[fill=black] (1.5,0.866025) circle (1.5pt);
\draw[fill=black] (2.5,0.866025) circle (1.5pt);
\draw[fill=black] (3.5,0.866025) circle (1.5pt);
\draw[fill=black] (0,1.732050) circle (1.5pt);
\draw[fill=black] (1,1.732050) circle (1.5pt);
\draw[fill=black] (2,1.732050) circle (1.5pt);
\draw[fill=black] (3,1.732050) circle (1.5pt);
\draw[fill=black] (0.5,2.598076) circle (1.5pt);
\draw[fill=black] (1.5,2.598076) circle (1.5pt);
\draw[fill=black] (2.5,2.598076) circle (1.5pt);
\draw[fill=black] (3.5,2.598076) circle (1.5pt);
\draw[fill=black] (1.8,1.35) circle (0.75pt);
\draw (1.8,1.35) circle (1);
\end{tikzpicture}
\caption{Proof of Proposition~\ref{prop:compare}.}
\label{fig:boundary}
\end{figure}
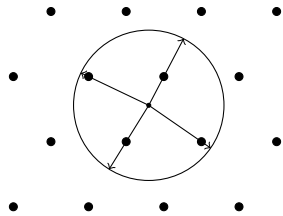

\begin{proof}[Proof of Proposition~\ref{prop:compare}]
See Figure~\ref{fig:boundary}. Suppose we have a packing of unit
spheres in $\RR^n$ with density $\Delta$. Let $S_R^{n-1}$ be a sphere
in $\RR^n$ of radius $R \leq 2$ (to be chosen later), located so that
it contains at least $\Delta R^n$ of the centers of the spheres in the
packing but its center is not one of them.  Such a location always
exists, by the same averaging argument as above (a randomly chosen location will
contain an average of $\Delta R^n$ sphere centers). Now, project the
sphere centers from the packing onto the surface of $S_R^{n-1}$ using
rays starting from the center of $S_R^{n-1}$. It follows from the lemma
below that the projections are separated by angles of at least
$\theta$, where $\sin(\theta/2) = 1/R$. Therefore, $\Delta R^n \leq
A(n, \theta)$, as desired, and we can achieve any angle of $\pi/3$
or more using $R \le 2$.
\end{proof}

Note that the proof breaks down if $R>2$, because two projected sphere
centers can even coincide.

\begin{lemma} \label{lem:angle}
Suppose $R \le 2$. If $XYZ$ is a triangle with $\abs{XY} \geq 2$,
$\abs{XZ} \leq R$, $\abs{YZ} \leq R$, then $\angle{XZY} \geq \theta$,
where $\sin(\theta/2) = 1/R$.
\end{lemma}

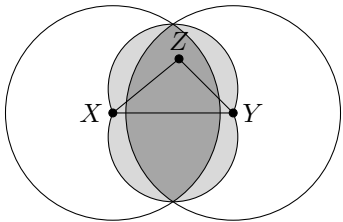
\begin{figure}
\centering
\tikzstyle{v} = [circle, draw, fill=black, inner sep=0pt, minimum width=3pt]
\begin{tikzpicture}[scale=0.8]
\path[use as bounding box] (-2,-2) rectangle (4,2);
\coordinate[label=left:$X$] (X) at (0,0) {};
\coordinate[label=right:$Y$] (Y) at (2,0) {};

\def\yyt{.4}
\pgfmathparse{sqrt(1+\yyt*\yyt)}\global\let\rri\pgfmathresult
\pgfmathparse{\rri+\yyt}\global\let\yyz\pgfmathresult
\pgfmathparse{sqrt(1+\yyz*\yyz)}\global\let\rrz\pgfmathresult

\begin{scope}
\clip (-3,0) rectangle (3,-3);
\draw[fill=gray!30] (1,-\yyt) circle (\rri);
\end{scope}

\begin{scope}
\clip (-3,0) rectangle (3,3);
\draw[fill=gray!30] (1,\yyt)  circle (\rri);
\end{scope}

\begin{scope}
\clip (X) circle (\rrz);
\path[fill=gray!70] (Y) circle (\rrz);
\end{scope}

\draw (X) circle (\rrz);
\draw (Y) circle (\rrz);
\coordinate[label=above:$Z$] (Z) at (1.1,0.9) {};

\node[v] at (X) {};
\node[v] at (Y) {};
\node[v] at (Z) {};
\draw (X) -- (Y) -- (Z) -- (X);
\end{tikzpicture}

\caption{Pictorial proof of Lemma~\ref{lem:angle}.  The bounds $\abs{XZ} \leq R$ and $\abs{YZ}
\leq R$ place $Z$ in the dark gray region, which is the intersection of
the two disks centered at $X$ and $Y$ with radius $R \leq 2$. The light
gray region contains all points $P$ with $\angle XPY \geq \theta$.
Since the dark region is contained inside the light region, it follows
that $\angle XZY \geq \theta$.}
\label{fig:intersect}
\end{figure}

\begin{proof}
See Figure~\ref{fig:intersect} for a pictorial proof.
For an algebraic proof, let $x = \abs{XZ}$, $y =
\abs{YZ}$, $z = \abs{XY}$, and $\gamma = \angle XZY$. By the law of
cosines, $\cos\gamma = (x^2 + y^2 - z^2)/(2xy)$. By taking partial
derivatives, we see that the expression $(x^2 + y^2 - z^2)/(2xy)$ is
maximized in the domain $0 \leq x,y \leq R$ and $z \geq 2$ at $(x,y,z)
= (R,R,2)$. Therefore, $\cos\gamma \leq 1 - 2R^{-2} = 1 - 2
\sin^2(\theta/2) = \cos\theta$. It follows that $\gamma \geq \theta$.
\end{proof}

Inequalities~\eqref{eq:lift} and~\eqref{eq:prop-compare} can be stated a
little more naturally in terms of packing density on the sphere.  A spherical
code on $S^{n-1}$ with minimal angle $\theta$ and size $A(n,\theta)$
corresponds to a packing with spherical caps of angular radius $\theta/2$
with density
\begin{equation} \label{eq:spheredensity}
A(n,\theta) \frac{\int_0^{\theta/2} \sin^{n-2} x \, dx}{\int_0^{\pi} \sin^{n-2} x \, dx}.
\end{equation}
In other words, it covers this fraction of the sphere.  Now if we let
$\Delta_{S^{n-1}}(\theta)$ denote the optimal packing density, then
\eqref{eq:prop-compare} implies
\begin{equation}
  \label{eq:euclidean-spherical}
\frac{1}{n} \log \Delta_{\RR^n} \lesssim \frac{1}{n} \log \Delta_{S^{n-1}}(\theta),
\end{equation}
where $f(n) \lesssim g(n)$ means $f(n) \le h(n)$ for some function $h$ with
$h(n) \sim g(n)$ (i.e., $\lim_{n \to \infty} h(n)/g(n) = 1$). This simply
amounts to verifying that
\[
\frac{1}{n} \log \frac{\int_0^{\theta/2} \sin^{n-2} x \, dx}{\int_0^{\pi} \sin^{n-2} x \, dx}
\sim \log \sin \frac\theta2
\]
for fixed $\theta$ satisfying $0 < \theta \le \pi$. Furthermore, it is known
that
\begin{equation}
  \label{eq:sphere-monotone}
\frac{1}{n} \log \Delta_{S^{n-1}}(\theta) \lesssim \frac{1}{n+1} \log \Delta_{S^{n}}(\phi)
\end{equation}
for $0 < \theta < \phi \le \pi/2$ (see (17) in \cite{L75}).  Thus, the exponential
rate of the packing density for spherical caps is weakly increasing as
a function of angle, and Euclidean space naturally occurs as the zero
angle limit.

The proof of the Kabatiansky-Levenshtein bound \eqref{eq:KL} on
$\Delta_{\RR^n}$ uses the following bound on $A(n, \theta)$ for $0 <
\theta < \pi/2$, which is derived using the linear programming bound
for spherical codes (see Theorem~4 in \cite{KL78}):
\begin{equation}
\label{eq:spherical-code-bound}
\frac{1}{n} \log A(n,\theta)
\lesssim \frac{1+\sin\theta}{2\sin\theta} \log \frac{1+\sin\theta}{2\sin\theta}
- \frac{1 - \sin\theta}{2\sin\theta} \log \frac{1 - \sin\theta}{2 \sin\theta}.
\end{equation}
The bound~\eqref{eq:KL} is then deduced by setting
\eqref{eq:spherical-code-bound} into \eqref{eq:lift} and choosing
$\theta$ to minimize the resulting bound,\footnote{Let us clarify a potentially
confusing point. The fact that $\theta = 1.0995\ldots$ minimizes the
bound may, at first, seem to be at odds with
\eqref{eq:euclidean-spherical} and \eqref{eq:sphere-monotone}, where we
said that the exponential rate of the packing density
$\Delta_{S^{n-1}}(\theta)$ is weakly increasing in $\theta$. Both
statements are correct. The bound in \eqref{eq:spherical-code-bound} is
a preliminary bound on $A(n,\theta)$, which can be improved for
$\theta$ less than the critical value $1.0995\ldots$ by
incorporating \eqref{eq:sphere-monotone}.  This improvement yields the same
bound on $\Delta_{\RR^n}$ for all $\theta \le 1.0995\ldots$.} which turns out to happen at
$\theta = 1.0995\ldots \approx 0.35\pi$. If we now apply our new inequality
\eqref{eq:prop-compare} in place of \eqref{eq:lift}, then we obtain an
improvement in the bound by a factor of $\mathcal{A}_{n+1}/\mathcal{A}_n$, where $\mathcal{A}_n = (1.2635\ldots+o(1))^n$
is the Kabatiansky-Levenshtein bound on $A(n,1.0995\ldots)$.
Thus, we obtain an improved sphere packing bound by a factor of $1.2635\ldots$
on average, in the sense that the geometric mean of the improvement
factors over all dimensions from $1$ to $N$ tends to $1.2635\ldots$ as $N \to \infty$.

\section{Linear programming bounds} \label{sec:LP}

In \cite{KL78} the upper bound on the maximum sphere packing density
$\Delta_{\RR^n}$ was derived by first giving an upper bound for the
maximum size $A(n, \theta)$ of a spherical code using linear
programming, and then using \eqref{eq:lift} to compare the two
quantities. We refer to this method as the Kabatiansky-Levenshtein
approach. Cohn and Elkies \cite{CE03} took a more direct approach to
bounding $\Delta_{\RR^n}$, by setting up a different linear program. In
this section, we show that the Cohn-Elkies linear program can always
prove at least as strong a bound on $\Delta_{\RR^n}$ as the
Kabatiansky-Levenshtein approach.

This theorem is the continuous analogue of a theorem of Rodemich
\cite{R80} in coding theory (see Theorem~3.5 of \cite{D94} for a proof
of Rodemich's theorem, since Rodemich published only an abstract). Let
$A(n,d)$ denote the maximum size of a binary error-correcting code of
block length $n$ and minimal Hamming distance $d$ (i.e., a subset of
$\set{0,1}^n$ with every two elements differing in at least $d$
positions), and let $A(n,d,w)$ denote the maximum size of such a code
with constant weight $w$ (i.e., every element of the subset has exactly
$w$ ones).  The current best bounds on $A(n,d)$ and $A(n,d,w)$ for
large $n$ are by McEliece, Rodemich, Rumsey, and Welch \cite{MRRW},
using linear programming bounds. As in the Kabatiansky-Levenshtein
approach, some of the best bounds on $A(n,d)$ were obtained using
bounds on $A(n,d,w)$ along with an analogue of
Proposition~\ref{prop:compare} known as the Bassalygo-Elias
inequality~\cite{Bas65}:
\begin{equation}
\label{eq:BE}
A(n,d) \leq \frac{2^n}{\binom{n}{w}} A(n,d,w).
\end{equation}
The proof of \eqref{eq:BE} is by an easy averaging argument. In analogy
with sphere packing, error-correcting codes play the role of sphere
packings while constant weight codes play the role of spherical codes.
Rodemich proved that any upper bound on $A(n,d)$ obtained using the
linear programming bound for $A(n,d,w)$ combined with \eqref{eq:BE} can
be obtained directly via the linear programming bound for $A(n,d)$.
Theorem~\ref{thm:LP-match} below is the continuous analogue of
Rodemich's theorem.

\subsection{LP bounds for spherical codes}

We begin by reviewing the linear programming bounds for spherical
codes.  We follow the approach of Kabatiansky and Levenshtein
\cite{KL78}, based on their inequality on the mean.

Let $S^{n-1}$ denote the unit sphere in $\RR^n$.  A function $f \colon
[-1,1] \to \RR$ is \emph{positive definite} if for all $N$ and all
$x_1,\dots,x_N \in S^{n-1}$, the matrix $\big(f(\ang{x_i,x_j})\big)_{1
\le i,j \le N}$ is positive semidefinite.  (Note that this property
depends on the choice of $n$; when necessary for clarity, we will say
such a function is \emph{positive definite on $S^{n-1}$}.)
Equivalently, for all $x_1,\dots,x_N \in S^{n-1}$ and $t_1,\dots,t_N
\in \RR$,
\[
 \sum_{1 \leq i, j \leq N} t_i t_j f(\ang{x_i,x_j}) \geq 0.
\]
A result of Schoenberg~\cite{Sch42} characterizes continuous
positive-definite functions as the nonnegative linear combinations of
the Gegenbauer polynomials $C_k^{n/2-1}$ for $k = 0, 1, 2, \dots$.
Recall that the polynomials $C_k^\alpha$ are orthogonal with respect to
the measure $(1-t^2)^{\alpha-1/2} \, dt$ on $[-1,1]$.  When $\alpha=n/2-1$,
this measure arises naturally (up to scaling) as the orthogonal projection of the
surface measure from $S^{n-1}$ onto a coordinate axis.

Given a positive-definite function $g$, define $\os{g}$ to be its
average
\[
\os{g} = \frac{\int_{-1}^1 g(t) (1-t^2)^{(n-3)/2} \, dt}{\int_{-1}^1 (1-t^2)^{(n-3)/2} \, dt}
\]
with respect to this measure. Equivalently, $\os{g}$ is the expectation
of $g(\ang{x, y})$ with $x$ and $y$ chosen independently and uniformly
at random from $S^{n-1}$. If
\[
g(t) = \sum_{k \ge 0} c_k C_k^{n/2-1}(t),
\]
then $\os{g} = c_0$.

\begin{theorem}[Delsarte-Goethals-Seidel \cite{DGS77}, Kabatiansky-Levenshtein \cite{KL78}] \label{thm:DGS}
If $g \colon [-1,1] \to \RR$ is continuous and positive definite on
$S^{n-1}$, $g(t) \le 0$ for all $t \in [-1,\cos \theta]$, and $\os{g} >
0$, then
\[
A(n, \theta) \leq \frac{g(1)}{\os g}.
\]
\end{theorem}

Let $A^{\rm LP}(n,\theta)$ denote the best upper bound on $A(n,\theta)$
that could be derived using Theorem~\ref{thm:DGS}. In other words, it
is the infimum of $g(1)/\os{g}$ over all valid auxiliary functions $g$.

We will give a proof of this theorem
following the approach of \cite{KL78}, as preparation for giving a new
proof of Theorem~\ref{thm:CE} below.

\begin{proof}
Let $\mathcal{C}$ be any spherical code in $S^{n-1}$ with minimal angle
at least $\theta$, let $\mu$ be the surface measure on $S^{n-1}$,
normalized to have total measure $1$, let $\delta_x$ be a delta
function at the point $x$, and let
\[
\nu = \sum_{x \in \mathcal{C}} \delta_{x} + \lambda \mu,
\]
where $\lambda$ is a constant to be determined.  We have
\[
\iint g(\ang{x,y}) \, d\nu(x) \, d\nu(y) \ge 0,
\]
because we can approximate the integral with a sum and use the positive
definiteness of $g$.  This inequality amounts to
\[
\lambda^2 \os{g} + 2 \lambda |\mathcal{C}| \os{g} + \sum_{x,y \in \mathcal{C}} g(\ang{x,y}) \ge 0.
\]
Because $\ang{x,y} \le \cos \theta$ for distinct points $x,y \in
\mathcal{C}$ and $g(t) \leq 0 $ for $t \in [-1, \cos \theta]$,
we have
\[
\sum_{x,y \in \mathcal{C}} g(\ang{x,y}) \le \sum_{x \in \mathcal{C}} g(\ang{x,x}) = |\mathcal{C}|g(1).
\]
Thus,
\[
\lambda^2 \os{g} + 2 \lambda |\mathcal{C}| \os{g} +  |\mathcal{C}|g(1) \ge 0.
\]

To derive the best bound on $|\mathcal{C}|$, we take $\lambda =
-|\mathcal{C}|$.  Then
\[
0 \le -|\mathcal{C}|^2 \os{g} + |\mathcal{C}| g(1)
\]
and hence
\[
|\mathcal{C}| \le\frac{g(1)}{\os g},
\]
as desired.
\end{proof}

\subsection{LP bounds in Euclidean space}

The Kabatiansky-Levenshtein approach gives the following bound on
$\Delta_{\RR^n}$. The original version uses \eqref{eq:lift}, but here
we state the improved version using Proposition~\ref{prop:compare}.

\begin{corollary} \label{cor:KL}
Suppose $g$ satisfies the hypotheses of Theorem~\ref{thm:DGS} with
$\pi/3 \le \theta \le \pi$. Then
\[
\Delta_{\RR^n} \leq \sin^n(\theta / 2) \frac{g(1)}{\ol g}.
\]
\end{corollary}

Let us recall the Cohn-Elkies linear programming bound.  Given an integrable
function $f \colon \RR^n \to \RR$, let $\wht f$ denote its Fourier
transform, normalized by
\[
\widehat{f}(t) = \int_{\RR^n} f(x) e^{2\pi i \langle x,t \rangle} \, dt.
\]
Let $B_R^n$ denote the $n$-dimensional ball with radius $R$. The volume
of the $n$-dimensional unit ball is $\vol(B_1^n) = \pi^{n/2}/(n/2)!$,
where $(n/2)! = \Gamma(n/2 + 1)$ for $n$ odd.

Much like the case of spheres, a function $f \colon \RR^n \to \RR$ is
\emph{positive definite} if for all $N$ and all $x_1,\dots,x_N \in
\RR^n$, the matrix $\big(f(x_i-x_j)\big)_{1 \le i,j \le N}$ is positive
semidefinite.  A result of Bochner~\cite{Boc33} characterizes
continuous positive-definite functions as the Fourier transforms of
finite Borel measures.  If $f$ and $\wht f$ are both integrable, then
$f$ is positive definite if and only if $\wht f$ is nonnegative
everywhere, by Fourier inversion and Bochner's theorem.

\begin{theorem}[Cohn-Elkies \cite{CE03}] \label{thm:CE}
Suppose $f \colon \RR^n \to \RR$ is continuous, positive definite, and
integrable, $f(x) \le 0$ for all $|x| \ge 2$, and $\wht f(0) > 0$. Then
\[
\Delta_{\RR^n} \leq \vol(B_1^n) \frac{f(0)}{\wht f(0)}.
\]
\end{theorem}

The original version in \cite{CE03} required suitable decay of $f$ and
$\wht f$ at infinity, and it was based on Poisson summation.  These
more restrictive hypotheses were removed in Section~9 of \cite{CK07}.
Here we give a more direct proof, although it has the disadvantage of
not telling as much about what happens when equality holds as the
Poisson summation proof does.

\begin{proof}
Without loss of generality, we can symmetrize to assume $f$ is an even
function (indeed, radially symmetric).  This is not necessary for the
proof, but it will simplify some of the expressions below.

Let $\mathcal{P}$ be a packing with balls of radius $1$, such that
$\mathcal{P}$ has density $\Delta_{\RR^n}$.  Given a radius $r>0$, let
$S_r$ be the set of sphere centers from $\mathcal{P}$ that lie within
the ball of radius $r$ about the origin, let $V_r$ be the volume of
that ball, and let $N_r = |S_r|$.  Then
\[
\lim_{r \to \infty} \vol(B_1^n) \frac{N_r}{V_r} = \Delta_{\RR^n}.
\]

Let $R = r + \sqrt{r}$ (in fact, $\sqrt{r}$ could be replaced with any
function that tends to infinity but is $o(r)$). Consider the signed
measure
\[
\nu = \sum_{x \in S_r} \delta_x + \lambda \mu_R,
\]
where $\delta_x$ is the delta function at $x$, $\mu_R$ is Lebesgue
measure on the ball of radius $R$ centered at the origin, and $\lambda$
is a constant to be determined. As in the proof of
Theorem~\ref{thm:DGS},
\[
\iint f(x-y) \, d\nu(x) \, d\nu(y) \ge 0,
\]
because $f$ is positive definite.  Equivalently,
\[
\lambda^2 \iint_{|x|,|y| \le R} f(x-y) \, dx \, dy
+ 2\lambda \sum_{x \in S_r} \int_{|y| \le R} f(x-y) \, dy
+ \sum_{x,y \in S_r} f(x-y) \ge 0.
\]
Because $f(x-y) \le 0$ whenever $x$ and $y$ are distinct points in the
packing,
\[
\lambda^2 \iint_{|x|,|y| \le R} f(x-y) \, dx \, dy
+ 2\lambda \sum_{x \in S_r} \int_{|y| \le R} f(x-y) \, dy
+ N_r f(0) \ge 0.
\]
Assuming $r$ is large enough that $N_r>0$, we set $\lambda = -N_r/V_r$ and divide by $N_r$ to obtain
\[
\frac{N_r}{V_r} \cdot \frac{1}{V_r} \iint_{|x|,|y| \le R} f(x-y) \, dx \, dy
- 2\cdot \frac{N_r}{V_r} \cdot \frac{1}{N_r} \sum_{x \in S_r} \int_{|y| \le R} f(x-y) \, dy
+ f(0) \ge 0.
\]

It is not hard to compute the limits
\[
\lim_{r \to \infty} \frac{1}{V_r} \iint_{|x|,|y| \le R} f(x-y) \, dx \, dy = {\wht f}(0)
\]
and
\[
\lim_{r \to \infty} \frac{1}{N_r}\sum_{x \in S_r} \int_{|y| \le R} f(x-y) \, dy = {\wht f}(0).
\]
Specifically, when $|x| \le r$, the $y$-integral covers all
values of $x-y$ up to radius $R-r = \sqrt{r}$.  As $r \to \infty$ these
$y$-integrals converge to $\wht f(0)$, and all but a negligible
fraction of the values of $x$ satisfying $|x| \le R$ also satisfy $|x|
\le r$.

Thus, in the limit as $r \to \infty$ we find that
\[
\frac{\Delta_{\RR^n}}{\vol(B_1^n)} {\wht f}(0)
- 2\frac{\Delta_{\RR^n}}{\vol(B_1^n)} {\wht f}(0)
+ f(0) \ge 0,
\]
which is equivalent to the desired inequality.
\end{proof}

Let $\Delta_{\RR^n}^{\rm LP}$ denote the optimal upper bound on
$\Delta_{\RR^n}$ using Theorem~\ref{thm:CE}. Recall that $A^{\rm LP}(n,
\theta)$ denotes the optimal upper bound on $A(n,\theta)$ obtained
using Theorem~\ref{thm:DGS}. Our next result compares the LP bound on
the sphere packing density $\Delta_{\RR^n}$ obtained from
Corollary~\ref{cor:KL} with the one from Theorem~\ref{thm:CE}.

\begin{theorem} \label{thm:LP-match}
For $\pi/3 \le \theta \le \pi$ and positive integers $n$,
\[
\Delta_{\RR^n}^{\rm LP} \leq \sin^n(\theta / 2) A^{\rm LP}(n, \theta).
\]
\end{theorem}

To prove Theorem~\ref{thm:LP-match}, we will show that for any upper
bound on $\Delta_{\RR^n}$ obtained using a function $g$ in
Corollary~\ref{cor:KL}, we can always find a function $f$ that gives a
matching bound using Theorem~\ref{thm:CE}. In other words,
\[
\sin^n(\theta / 2)
  \frac{g(1)}{\ol g} =  \vol(B_1^n) \frac{f(0)}{\wht f(0)}.
\]

We have a similar conclusion for the original Kabatiansky-Levenshtein
bound using \eqref{eq:lift} without the $\theta \geq \pi/3$
assumption. See the remarks following the proof.

\begin{proof}[Proof of Theorem~\ref{thm:LP-match}]
Let $g$ be any function satisfying the hypotheses of
Theorem~\ref{thm:DGS}. The idea is to construct a function $f \colon
\RR^n \to \RR$ based on $g$ mimicking the geometric argument in the
proof of Proposition~\ref{prop:compare}. Let $R = 1/\sin(\theta/2)$, as
in that proof.

Consider the integral
\[
\int_{B_R^n(x) \cap B_R^n(y)}
g\paren{\ang{\frac{x-z}{\abs{x-z}}, \frac{y-z}{\abs{y-z}}}} \, dz,
\]
where $B_R^n(x)$ is the ball of radius $R$ centered at $x$.
Note that
\[
\ang{\frac{x-z}{\abs{x-z}}, \frac{y-z}{\abs{y-z}}} = \cos \angle xzy,
\]
where $\angle xzy$ denotes the angle at $z$ formed by $x$ and
$y$.  This angle is not defined if $x=z$ or $y=z$, but these
cases occur with measure zero.

The integral depends only on $|x-y|$, so there is a radial
function $f \colon \RR^n \to \RR$ satisfying
\[
f(x-y) = \int_{B_R^n(x) \cap B_R^n(y)}
g\paren{\ang{\frac{x-z}{\abs{x-z}}, \frac{y-z}{\abs{y-z}}}} \, dz.
\]
We claim that $f$ is positive definite. Indeed, let $\chi_R$ denote the
characteristic function of $B_R^n(0)$. Then we can rewrite $f$ as
\[
f(x-y)  =  \int_{\RR^n} \chi_R(x-z)\chi_R(y-z)
\,g\paren{\ang{\frac{x-z}{\abs{x-z}}, \frac{y-z}{\abs{y-z}}}} \, dz.
\]
For any $x_1, \dots, x_N \in \RR^n$ and $t_1, \dots, t_N \in \RR$, we
can expand
\[
\sum_{1 \leq i, j \leq N} t_it_j f(x_i-x_j)
\]
as
\[
\int_{\RR^n} \sum_{1 \leq i, j \leq N} (t_i \chi_R(x_i -z)) (t_j
\chi_R(x_j - z)) \, 
g\paren{\ang{\frac{x_i-z}{\abs{x_i-z}}, \frac{x_j-z}{\abs{x_j-z}}}} \,
dz.
\]
This expression is nonnegative, because $g$ is positive definite on the
unit sphere in $\RR^n$ and we can use $t_i\chi_R(x_i-z)$ as
coefficients. This shows that $f$ is positive definite on $\RR^n$.  It
is also integrable, because it has compact support (it vanishes past
radius $2R$).

If $\abs{x-y} \geq 2$, then by Lemma~\ref{lem:angle},
\[
\ang{\frac{x-z}{\abs{x-z}}, \frac{y-z}{\abs{y-z}}} \leq \cos \theta
\]
for all $z \in B_R^n(x) \cap B_R^n(y) \setminus\set{x,y}$. Since $g(t) \leq
0$ for all $t \in [-1, \cos\theta]$ by hypothesis, it follows that $f(x-y)
\leq 0$ whenever $\abs{x-y} \geq 2$. Thus, we have verified that $f$
satisfies all the hypotheses of Theorem~\ref{thm:CE} except
$\widehat{f}(0)>0$, which we will check shortly. We have $f(0) = \vol(B_R^n)
g(1)$ and
\begin{align*}
\wht f (0)
&= \int_{\RR^n} f(x-0) \, dx
\\&= \int_{\RR^n} \int_{\RR^n} \chi_R(x-z)\chi_R(-z)
\,g\paren{\ang{\frac{x-z}{\abs{x-z}}, \frac{-z}{\abs{-z}}}} \, dx \, dz
\\&= \int_{\RR^n} \int_{\RR^n} \chi_R(u)\chi_R(v)
\,g\paren{\ang{\frac{u}{\abs{u}}, \frac{v}{\abs{v}}}} \, du \,dv
\\&= \vol(B_R^n)^2 \ol g.
\end{align*}
Therefore $\widehat{f}(0)>0$ and
\[
\vol(B_1^n) \frac{f(0)}{\wht f(0)}
= \vol(B_1^n) \frac{\vol(B_R^n)}{\vol(B_R^n)^2} \frac{g(1)}{\ol g}
= \frac{1}{R^n} \frac{g(1)}{\ol g} = \sin^n(\theta/2) \frac{g(1)}{\ol g},
\]
as desired.
\end{proof}

When $\theta < \pi/3$ we can similarly match the Kabatiansky-Levenshtein
bound obtained using \eqref{eq:lift} by adapting the above proof for the
corresponding geometric argument. Let $\pi \colon B_1^n \to \{x \in S^n :
x_{n+1} \ge 0\}$ denote the map that orthogonally projects the unit disk in
the hyperplane $\RR^n \times \set{0}$ in $\RR^{n+1}$ to the upper half of the
unit sphere in $\RR^{n+1}$. For any $g$ in Theorem~\ref{thm:DGS} that gives a
bound for $A(n+1, \theta)$, let
\[
f(x-y) = \int_{B_R^n(x) \cap B_R^n(y)}
g\paren{\ang{\pi\paren{\frac{x-z}{R}},
\pi\paren{\frac{y-z}{R}}}} \, dz.
\]
A similar argument shows that $f$ is positive definite and $f(x) \leq 0$
whenever $\abs{x} \geq 2$. We have $f(0) = \vol(B_R^n) g(1)$ and $\wht f(0) =
\vol(B_R^n)^2 \EE[g(\ang{\pi(u), \pi(v)})]$, where $u$ and $v$ are
independent uniform random points in $B_1^n$.  The inequality on the mean
from \cite{KL78} says that the average of a positive-definite kernel with respect
to a probability distribution on its inputs must be at least as large as that
with respect to the uniform distribution.  Thus, $\EE[g(\ang{\pi(u), \pi(v)})] \ge
\ol g$ and
\[
\vol(B_1^n) \frac{f(0)}{\wht f(0)}
\le \sin^n(\theta/2) \frac{g(1)}{\ol g}.
\]
However, we cannot conclude that $\wht f(0) =
\vol(B_R^n)^2 \ol g$, so the version of this argument in
Theorem~\ref{thm:LP-match} is more elegant.

\section{Hyperbolic sphere packing}
\label{sec:hyperbolic}

Hyperbolic sphere packing is far more subtle than Euclidean sphere
packing.  In both hyperbolic and Euclidean spaces, one must deal with
the infinite volume of space available.  The Euclidean solution is
fairly straightforward: restrict to a large but bounded region, and
then let the size of this region tend to infinity.  The boundary
effects have negligible influence on the global density.  However,
these arguments become much trickier in hyperbolic space, since the
exponential volume growth means the limiting behavior is dominated by
what happens near the boundary. Troubling phenomena occur, such as
packings that have different densities when one uses regions centered
at different points. There are numerous other pathological examples
(see, for example, Section~1 of \cite{BR04}), and it is only recently
that a widely accepted definition of density has been proposed by Bowen
and Radin \cite{BR03,BR04}.  Before this definition, some density
bounds were proved using Voronoi cell arguments that would apply to any
reasonable definition of density, and indeed they apply to the
Bowen-Radin definition (see Proposition~3 in \cite{BR03}).

The best bound known is due to B\"or\"oczky \cite{B78}, who gave an
upper bound for the fraction of each Voronoi cell that could be covered
in a hyperbolic sphere packing.  The bound depends on the radius of the
spheres in the packing (the curvature of hyperbolic space sets a
distance scale, so density is no longer scaling-invariant, as it is in
Euclidean space). At least in sufficiently high dimensions, the B\"or\"oczky bound
is an increasing function of radius \cite{M99}, so it is never
better than the radius-zero limit.  In that limit it degenerates to the
Rogers bound \cite{Rog58}, which in dimension $n$ is asymptotic to $2^{-n/2} \cdot n/e$
as $n \to \infty$.

Here, we improve the density bound to the Kabatiansky-Levenshtein bound,
regardless of the radius. Let $\Delta_{\HH^n}(r)$ denote the optimal packing
density for balls of radius $r$ in $\HH^n$ (we will define this density
precisely in Section~\ref{subsec:bowen-radin}). We can bound the packing
density of balls in hyperbolic space by the packing density of spherical caps
on a sphere, as in the Euclidean setting discussed in
Section~\ref{sec:geometric}. The next result is analogous to
Proposition~\ref{prop:compare}.

\begin{theorem}
\label{thm:hyperbolic-compare} For all $n \geq 2$,  $\pi/3 \leq \theta \leq
\pi$,  and $r
> 0$, we have
\[
\Delta_{\HH^n}(r) \leq \sin^{n-1}(\theta/2) A(n,\theta).
\]
\end{theorem}

More precisely, one could replace $\sin^{n-1}(\theta/2)$ with the
hyperbolic volume ratio $\vol(B_r^n)/\vol(B_R^n)$, where $R$ is defined
by $\sinh R = (\sinh r)/\sin(\theta/2)$.  That would slightly improve
the inequality without changing the proof, at the cost of making the
statement more cumbersome.

As in the Euclidean case \eqref{eq:euclidean-spherical}, this theorem
implies that
\begin{equation}
\label{eq:hyperbolic-spherical}
\sup_{r > 0} \frac{1}{n} \log \Delta_{\HH^n} (r) \lesssim \frac{1}{n} \log \Delta_{S^{n-1}}(\theta).
\end{equation}
By using the Kabatiansky-Levenshtein bound on $\Delta_{S^{n-1}}$, i.e.,
\eqref{eq:spherical-code-bound} with $\theta \approx 0.35\pi$, we obtain the
following new bound on $\Delta_{\HH^n}(r)$.  It is an exponential improvement
over the B\"or\"oczky bound, which was previously the best bound known, and the
new bound is independent of the radius of the balls used in the packing.

\begin{corollary}
\label{cor:hyperbolic-KL} We have
\[
\sup_{r > 0} \Delta_{\HH^n}(r) \leq 2^{-(0.5990\ldots + o(1))n}.
\]
\end{corollary}

\subsection{The Bowen-Radin theory of hyperbolic packings}
\label{subsec:bowen-radin}

The Bowen-Radin approach to hyperbolic packing is based on ergodic theory,
but our argument is elementary.  All we need is the following fact: for every
$R>0$, there exists a ball $\mathcal{B}$ of radius $R$ containing a subset of
at least
\[
\Delta_{\HH^n}(r) \frac{\vol(B^n_R)}{\vol(B^n_r)}
\]
points at distance $2r$ or more from each other and not equal to the center
of $\mathcal{B}$.  Naively, this should follow from a simple averaging
argument, since if we place $\mathcal{B}$ at random in a dense packing, then
this is the expected number of sphere centers it will contain, and the
probability that one of them will hit the center of $\mathcal{B}$ is zero.
Before turning to the proof of Theorem~\ref{thm:hyperbolic-compare}, we will
briefly explain the Bowen-Radin definition and why this fact is true.

In the Bowen-Radin theory, instead of focusing on individual packings one
studies measures on the space of packings.  Let $S_r$ be the space of
relatively dense packings of $\HH^n$ with balls of radius $r$ (i.e., packings
in which any additional such ball would intersect one from the packing).
Bowen and Radin give a natural metric to $S_r$, under which it is compact,
and they study the action of the isometry group $G$ of $\HH^n$ on $S_r$. They
define random packings by $G$-invariant Borel probability measures $\mu$ on
$S_r$, and they define the density of $\mu$ to be the probability that some
fixed origin is contained in one of the balls in the packing (by
$G$-invariance, it is independent of the choice of origin). The optimal
packing density $\Delta_{\HH^n}(r)$ is defined to be the least upper bound
for the density of such measures.

Although restricting attention to $G$-invariant measures may sound overly limiting,
it encompasses the reasonable examples that were known before.  For example,
if a packing is invariant under a discrete subgroup of $G$ with finite covolume,
then the Haar measure on $G$ descends to a probability distribution on the $G$-orbit
of the packing.  However, the space of measures is better behaved than the
space of discrete subgroups.

Bowen and Radin show that the optimal packing density
is achieved by some measure, and they show how to obtain well-behaved dense
sphere packings by sampling from such a distribution.  Their papers make a
convincing case that this ergodic approach is the right framework for studying
hyperbolic packing density.  See also \cite{R04} for intuition and
background.

The fact we need for Theorem~\ref{thm:hyperbolic-compare} is the following
lemma, which says that the sphere centers in a random packing are uniformly
distributed with point density $\delta/\vol(B^n_r)$:

\begin{lemma} \label{lemma:uniform}
Let $\mu$ be a $G$-invariant probability measure on $S_r$ with density
$\delta$. Then for every Borel set $A$ in $\HH^n$, the expected number of
sphere centers in $A$ for a $\mu$-random packing is
\[
\delta\frac{\vol(A)}{\vol(B^n_r)}.
\]
\end{lemma}

\begin{proof}
Let $\nu(A)$ be the expected number of sphere centers in a Borel set $A$.
Then $\nu$ is a $G$-invariant Borel measure on $\HH^n$, and the definition of
density can be reformulated as $\nu(B^n_r) = \delta$.  Thus, $\nu$ is locally
finite and therefore proportional to the hyperbolic volume measure.  (Recall
that Haar measure on $G/K$ is unique up to scaling, for any locally compact group
$G$ and compact subgroup $K$; see Chapter~III of \cite{N65}.)  The constant of
proportionality is determined by $\nu(B^n_r) = \delta$.
\end{proof}

\subsection{Proof of Theorem~\ref{thm:hyperbolic-compare}}

The proof of Theorem~\ref{thm:hyperbolic-compare} is analogous to the
Euclidean case. The heart of the proof is the following lemma.

\begin{lemma} \label{lem:hyp-proj}
Let $r \leq R \leq 2r$ and $\sin\frac\theta2 = \frac{\sinh r}{\sinh
R}$. In a packing of spheres of radius $r$ in $\HH^n$, every ball of
radius $R$ contains at most $A(n,\theta)$ sphere centers other than its
own center.
\end{lemma}

\begin{proof}
We use the same projection argument as in the proof of
Proposition~\ref{prop:compare}. Project the sphere centers from the
packing onto the surface of the ball of radius $R$ using rays starting
from the center of the ball. By the next lemma, the projections are
separated by angles of at least $\theta$, so there can be at most
$A(n,\theta)$ of them.
\end{proof}

The next lemma is the hyperbolic analogue of Lemma~\ref{lem:angle}.

\begin{lemma} \label{lem:hyp-angle}
Consider a hyperbolic triangle with side lengths $a, b, c$ and the
angle opposite to $c$ having measure $\gamma$. If $0 <a, b \leq R \leq
2r \leq c$, then
\[
\sin \frac{\gamma}{2} \geq \frac{\sinh r}{\sinh R}.
\]
\end{lemma}

\begin{proof}
By hyperbolic law of cosines,
\[
\cos \gamma = \frac{\cosh a \cosh b - \cosh c}{\sinh a \sinh b}.
\]
Let
\[
f(a,b,c) = \frac{\cosh a \cosh b - \cosh c}{\sinh a \sinh b}.
\]
We wish to maximize $f(a,b,c)$ in the domain $0 < a,b \leq R \leq 2r
\leq c$. Since $f$ is monotonically decreasing in $c$, it is maximized
by setting $c = 2r$. We have
\[
\frac{\partial f}{\partial a} = \frac{\cosh a \cosh c - \cosh b}{\sinh^2 a \sinh b}
\]
which is nonnegative since $\cosh c \geq \cosh b$ and $\cosh a \geq 1$.
Thus $f(a,b,c)$ is nondecreasing in $a$, and it is maximized by setting
$a = R$. The same is true for $b$ by symmetry, and so
\[
\cos\gamma = f(a,b,c) \leq \frac{\cosh^2 R - \cosh 2r}{\sinh^2 R}
= 1 - \frac{2\sinh^2 r}{\sinh^2 R}.
\]
Therefore
\[
\sin^2\frac{\gamma}{2} = \frac{1 - \cos\gamma}{2} \geq \frac{\sinh^2
r}{\sinh^2 R},
\]
and the result follows.
\end{proof}

\begin{proof}[Proof of Theorem~\ref{thm:hyperbolic-compare}]
Define $R$ to satisfy $\sin\frac\theta2 = \frac{\sinh r}{\sinh R}$.
Since $\pi/3 \le \theta \le \pi$, we have $r \leq R \leq
2r$. (Note that the inequality $R \le 2r$ does not always hold when $\theta <
\pi/3$.  It fails in the limit as $r \to 0$ but holds for large
$r$.)

Let $\mu$ be a Bowen-Radin measure with density $\Delta_{\HH^n}(r)$, and let
$A$ be a ball of radius $R$ with its center omitted.  By
Lemma~\ref{lemma:uniform}, the expected number of sphere centers in $A$ from
a $\mu$-random packing is $\Delta_{\HH^n}(r)
\frac{\vol(B^n_R)}{\vol(B^n_r)}$, and thus there exists a packing in which
there are at least this many. By Lemma~\ref{lem:hyp-proj},
\[
\Delta_{\HH^n}(r) \le \frac{\vol(B^n_r)}{\vol(B^n_R)} A(n,\theta),
\]
and so all that remains is to bound $\vol(B^n_r)/\vol(B^n_R)$.
The volume of a ball in $\HH^n$ is given by
\[
\vol(B^n_r) = \Omega_n \int_{0}^r \sinh^{n-1} x \ dx,
\]
where $\Omega_n = 2\pi^{n/2}/\Gamma(n/2)$ is the surface volume of the
unit Euclidean $(n-1)$-sphere. Thus
\begin{equation}
\label{eq:hyp-bound}
\begin{aligned}
\Delta_{\HH^n}(r) &\le \frac{\int_0^r \sinh^{n-1} x \, dx}{\int_0^R \sinh^{n-1} x \, dx}
A(n,\theta)\\
&\leq \paren{\frac{\sinh r}{\sinh R}}^{n-1} A(n,\theta)\\
&= \sin^{n-1}(\theta/2) A(n,\theta),
\end{aligned}
\end{equation}
where the second inequality follows from Lemma~\ref{lem:sinh-ratio} below.
\end{proof}

If we fix the ratio $\frac{\sinh r}{\sinh R}$, then the ratio of the
integrals in \eqref{eq:hyp-bound} is almost determined by the following
lemma (the lower bound is sharp as $r \to 0$ and the upper bound is
sharp as $r \to \infty$).  We do not need the lower bound, but it shows
that \eqref{eq:hyperbolic-spherical} cannot be substantially improved
by a more careful analysis of the volume of hyperbolic balls.

\begin{lemma} \label{lem:sinh-ratio}
For $0 < r \le R$,
\[
\left(\frac{\sinh r}{\sinh R}\right)^n \le
\frac{\int_0^r \sinh^{n-1} x \, dx}{\int_0^R \sinh^{n-1} x \, dx}
\le \left(\frac{\sinh r}{\sinh R}\right)^{n-1}.
\]
\end{lemma}

\begin{proof}
These inequalities amount to saying that
\[
\frac{\int_0^r \sinh^{n-1} x \, dx}{\sinh^{n-1} r}
\]
is an increasing function of $r$, while
\[
\frac{\int_0^r \sinh^{n-1} x \, dx}{\sinh^n r}
\]
is a decreasing function of $r$.

The derivative of the former function is
\[
1 - \frac{(n-1)\cosh r}{\sinh^n r} \int_0^r \sinh^{n-1} x \, dx,
\]
so we must prove that
\[
\frac{\sinh^n r}{(n-1) \cosh r} - \int_0^r \sinh^{n-1} x \, dx\ge 0.
\]
The left side of this inequality vanishes when $r=0$, and its
derivative with respect to $r$ is
\[
\frac{\sinh^{n-1} r}{(n-1) \cosh^2 r},
\]
so it is increasing and hence nonnegative.

To show that
\[
\frac{\int_0^r \sinh^{n-1} x \, dx}{\sinh^n r}
\]
is decreasing, note that its derivative is
\[
\frac{1}{\sinh r}- \frac{n \cosh r}{\sinh^{n+1} r}\int_0^r \sinh^{n-1} x \, dx,
\]
so we must prove that
\[
\int_0^r \sinh^{n-1} x \, dx - \frac{\sinh^n r}{n \cosh r}\ge 0.
\]
Again the left side vanishes when $r=0$, and this time its derivative
is
\[
\frac{\sinh^{n+1} r}{n \cosh^2 r},
\]
so it is increasing and hence nonnegative.  This completes the proof.
\end{proof}

\section{Linear programming bounds in hyperbolic space}
\label{sec:hypLP}

It is natural to try to extend the results of Section~\ref{sec:LP} on
linear programming bounds to hyperbolic space, but one runs into
technical difficulties.

Given a function $f \colon [0,\infty) \to \RR$, we view it as a
function of hyperbolic distance and define the corresponding kernel $f
\colon \HH^n \times \HH^n \to \RR$ by $f(x,y) = f(d(x,y))$, where $d$
denotes the metric on $\HH^n$.  (Using the same symbol for both
functions is an abuse of notation, but it is convenient not to have to
write the metric $d$ repeatedly, and the number of arguments makes it
unambiguous.)  We say $f$ is \emph{positive definite} if for all $N$ and
all $x_1,\dots,x_N \in \HH^n$, the matrix $\big(f(x_i,x_j)\big)_{1 \le
i,j \le N}$ is positive semidefinite, and we say it is
\emph{integrable} on $\HH^n$ if $x \mapsto f(x,y)$ is an integrable
function on $\HH^n$ (of course this is independent of $y$), in which
case we write $\int_{\HH^n} f$ for the integral.

Let $G$ be the connected component of the identity in the isometry
group of $\HH^n$, and let $K$ be the stabilizer within $G$ of a point
$e \in \HH^n$.  Then $(G,K)$ is a Gelfand pair; i.e., the algebra
$L^1(K\backslash G/K)$ of integrable, bi-$K$-invariant functions on $G$
forms a commutative algebra under convolution.  Here $G/K$ is $\HH^n$
and functions on $K \backslash G / K$ correspond to radial functions on
$\HH^n$.  See Chapters~8 and ~9 of \cite{W07} for an account of Gelfand
pairs and spherical transforms (and see \cite{T82} for a more concrete
exposition of Fourier analysis in $\HH^2$). In the setting of $\HH^n$,
this theory gives a well-behaved Fourier transform for radial
functions. For each $\lambda\ge0$, let $P_\lambda$ be the unique
radial eigenfunction of the Laplacian on $\HH^n$ with eigenvalue
$\lambda$ and $P_\lambda(0)=1$.  These functions are positive definite
for all $\lambda \ge 0$ (see Theorem~5.2 in \cite[p.~346]{T63}). Given
a function $f \colon [0,\infty) \to \RR$ that is integrable on $\HH^n$,
its radial Fourier transform is given by
\[
\widehat{f}(\lambda) = \int_{\HH^n} f(x,e) P_{\lambda}(x,e) \, dx,
\]
which is of course independent of $e \in \HH^n$.  As in the Euclidean
case, the Fourier transform extends to $L^2(K \backslash G / K)$, and
it yields an isomorphism from that space to $L^2([0,\infty),\mu_P)$,
where $\mu_P$ is the Plancherel measure. However, unlike the Euclidean
case, the Plancherel measure for $\HH^n$ is supported just on
$[(n-1)^2/4,\infty)$.

Positive-definite functions are characterized by the Bochner-Godement
theorem (see Theorems~9.3.4 and~9.4.1 in \cite{W07} or Theorem~12.10 in
Chapter~III of \cite{H08}). For continuous, integrable functions, it
says that $f$ is positive definite if and only
if $\widehat{f}$ is nonnegative on the support of the Plancherel
measure. However, $\widehat{f}$ can be negative outside of the support,
because $G$ is not amenable: Valette \cite{V98} has constructed a
continuous, positive-definite function with compact support and
negative integral. (His construction works in $G$, rather than $G/K$,
but it is easy to make it bi-$K$-invariant.)

In the linear programming bounds, we will assume $\widehat{f} \ge 0$
everywhere, which is a strictly stronger assumption than positive
definiteness.  We do not know whether the stronger hypothesis is truly
needed for the following conjecture, but it will be needed for the
proof of Theorem~\ref{thm:hypLP}.

\begin{conjecture} \label{conjecture:hypLP}
Let $f \colon [0,\infty) \to \RR$ be continuous and integrable on
$\HH^n$, and suppose $f(x) \le 0$ for all $x \ge 2r$ while
$\widehat{f}(\lambda) \ge0$ for all $\lambda>0$ and $\widehat{f}(0)>0$.
Then
\[
\Delta_{\HH^n}(r) \le \vol(B_r^n) \frac{f(0)}{\widehat{f}(0)}.
\]
\end{conjecture}

Here, of course, $\vol(B_r^n)$ denotes the volume of a ball of radius
$r$ in $\HH^n$.

Let
\[
\Delta^{\rm LP}_{\HH^n}(r) = \inf_f \vol(B_r^n) \frac{f(0)}{\widehat{f}(0)},
\]
where the infimum is over all $f$ satisfying the hypotheses of
Conjecture~\ref{conjecture:hypLP}.  The conjecture says $\Delta^{\rm
LP}_{\HH^n}(r)$ is an upper bound for $\Delta_{\HH^n}(r)$. Regardless of
whether that is true, $\Delta^{\rm LP}_{\HH^n}(r)$ can be viewed as the
solution of an abstract optimization problem.

The following theorem is the hyperbolic
analogue of Rodemich's theorem.

\begin{theorem} \label{theorem:hyperbolicRodemich}
For $\pi/3 \le \theta \le \pi$, positive integers $n \ge 2$, and $r>0$,
\[
\Delta^{\rm LP}_{\HH^n}(r) \le \sin^{n-1}(\theta/2) A^{\rm LP}(n,\theta).
\]
\end{theorem}

\begin{proof}
The argument is much like the proof of Theorem~\ref{thm:LP-match}.
Define $R$ by $\sinh R = (\sinh r)/\sin(\theta/2)$.  Given a
function $g$ satisfying the hypotheses of Theorem~\ref{thm:DGS}, define
$f$ by
\[
f(x,y) = \int_{B^n_R(x) \cap B^n_R(y)} g( \cos \angle xzy) \, dz,
\]
where $\angle xzy$ denotes the angle at $z$ formed by the geodesics to
$x$ and $y$.  Of course, this angle is not defined when $x=z$ or $y=z$,
but these cases occur with measure zero.

Exactly the same approach as in the proof of Theorem~\ref{thm:LP-match}
shows that $f$ is a positive-definite function and that $f(x,y) \le 0$
when $d(x,y) \ge 2r$. However, merely being positive definite does not
imply that $\widehat{f}(\lambda) \ge 0$ for all $\lambda \ge 0$.  To
prove that, we start by fixing $y \in \HH^n$ and writing
\begin{align*}
\widehat{f}(\lambda) &= \int_{\HH^n}
P_\lambda(x,y) f(x,y) \, dx\\
&= \int_{\HH^n} \int_{B^n_R(x) \cap B^n_R(y)} P_\lambda(x,y) g( \cos \angle xzy) \, dz \, dx\\
&= \int_{B^n_R(y)} \int_{B^n_R(z)}
P_\lambda(x,y) g( \cos \angle xzy) \, dx \, dz.
\end{align*}
The integrand depends only on $d(x,z)$, $d(y,z)$, and $\angle xzy$, because
the hyperbolic law of cosines determines $d(x,y)$ using this data, and the
integral is proportional to the expected value of $P_\lambda(x,y) g( \cos
\angle xzy)$ if we fix $y$, pick $z \in B^n_R(y)$ uniformly at random, and
then pick $x \in B^n_R(z)$.  Equivalently, we can fix $z$ and pick $x,y \in
B^n_R(z)$, because that induces the same measure on the three parameters
$d(x,z)$, $d(y,z)$, and $\angle xzy$.  This is obvious for $d(x,z)$ and
$d(y,z)$, since they simply follow the radial distance distribution on
$B^n_R$. (Picking $z \in B^n_R(y)$ or $y \in B^n_R(z)$ yields the same
distribution on $d(y,z)$.) For $\angle xzy$ it amounts to saying that the
angle at $z$ between a random point $x$ and a fixed point $y$ is distributed
the same as that between two random points $x$ and $y$.

Thus, we can change variables to fix $z$ instead of
$y$ and integrate over $x$ and $y$ to obtain
\[
\widehat{f}(\lambda) = \int_{B^n_R(z)} \int_{B^n_R(z)}
P_\lambda(x,y) g( \cos \angle xzy) \, dx \, dy.
\]
Now we can see that $\widehat{f}(\lambda) \ge 0$, because
$(x,y) \mapsto P_\lambda(x,y) g( \cos \angle xzy)$ defines a
positive-definite kernel for $x,y \in \HH^n\setminus\{z\}$.
Specifically, the product of two positive-definite kernels is
positive definite by the Schur product theorem (Theorem~7.5.3
in \cite{HJ}), which says that the set of positive-semidefinite
matrices is closed under the Hadamard product.

It also follows from this formula and $P_0=1$ that
\[
\widehat{f}(0) = \vol(B_R^n)^2 \ol{g},
\]
and combining this with $f(0) = \vol(B_R^n) g(1)$ yields
\[
\vol(B_r^n) \frac{f(0)}{\widehat{f}(0)} = \frac{\vol(B_r^n)}{\vol(B_R^n)} \frac{g(1)}{\ol g }
\le \sin^{n-1}(\theta/2) \frac{g(1)}{\ol g },
\]
as desired.
\end{proof}

In the remainder of this section, we explain why a straightforward
approach fails to prove Conjecture~\ref{conjecture:hypLP} and how to
prove it for periodic packings under an admissibility condition on $f$.
The latter proof is based on unpublished work of Cohn, Lurie, and Sarnak.

\subsection{Obstacles to proving the conjecture}
\label{subsec:obstacles}

We have not been able to prove Conjecture~\ref{conjecture:hypLP} by
imitating the proof of Theorem~\ref{thm:CE}.  The problem is that the
boundary effects when restricting to a ball are not negligible.

Specifically, consider a hyperbolic sphere packing with balls of radius
$r$, and imagine restricting it to a ball of radius $R$ (i.e., looking
only at the points within this large ball).  Let $S_R$ be the set of
all sphere centers in this ball and $\mu_R$ the hyperbolic volume measure on the
ball, and consider the signed measure
\[
\nu = \sum_{x \in S_R} \delta_x + \lambda \mu_R,
\]
where $\lambda$ is a constant to be specified shortly. Because $f$ is
positive definite,
\[
\iint f(x,y) \, d\nu(x) \, d\nu(y) \ge 0,
\]
which implies
\[
\lambda^2 \iint f(x,y) \, d\mu_R(x) \, d\mu_R(y) + 2 \lambda \sum_{x \in S_R}
\int f(x,y) \, d\mu_R(y) + |S_R| f(0) \ge 0,
\]
as in the proof of Theorem~\ref{thm:CE}.

Now suppose we have a Bowen-Radin measure on packings, with density $\delta$.
Averaging over such a measure yields
\[
\left(\lambda^2 + \frac{2\delta\lambda}{\vol(B_r^n)}\right) \iint f(x,y) \, d\mu_R(x) \, d\mu_R(y) + \delta \frac{\vol(B_R^n)}{\vol(B_r^n)} f(0) \ge 0,
\]
because Lemma~\ref{lemma:uniform} says the sphere centers are uniformly
distributed. In particular, taking $\lambda = - \delta/\vol(B_r^n)$ yields
\[
\delta \le \vol(B_r^n)\frac{f(0)}{\iint f(x,y) \, d\mu_R(x) \, d\mu_R(y) / \vol(B_R^n)}.
\]
This proves a legitimate bound on the density:

\begin{proposition}
Let $f \colon [0,\infty) \to \RR$ be continuous and positive definite
on $\HH^n$, and suppose $f(x) \le 0$ for all $x \ge 2r$.  Then for each
$R>0$,
\[
\Delta_{\HH^n}(r) \le \vol(B_r^n) \frac{f(0)}{\iint f(x,y) \, d\mu_R(x) \, d\mu_R(y) / \vol(B_R^n)},
\]
assuming the denominator is not zero.
\end{proposition}

It is natural to add the assumption that $f$ is integrable and then
take the limit as $R \to \infty$. However, the denominator does not
converge to $\widehat{f}(0)$, as it does in the Euclidean case.  To see
why, let $\chi_R$ be the characteristic function of $[0,R]$.
Then
\[
\iint f(x,y) \, d\mu_R(x) \, d\mu_R(y) = \int_{\HH^n} (f * \chi_R) \chi_R,
\]
where the right side denotes the integral of a radial function on $\HH^n$, and
the convolution is defined by
\[
(f * g)(x,y) = \int_{\HH^n} f(x,z) g(y,z) \, dz.
\]
Because the Fourier transform is unitary,
\[
\int_{\HH^n} (f * \chi_R) \chi_R = \int \widehat{f * \chi_R} \widehat{\chi}_R \, d \mu_P,
\]
where $\mu_P$ is the Plancherel measure, and that simplifies to
\[
\int \widehat{f} \widehat{\chi}_R^2 \, d \mu_P.
\]
One can show similarly that $\vol(B_R^n) = \int
\widehat{\chi}_R^2\, d \mu_P$, and thus
\[
\frac{\iint f(x,y) \, d\mu_R(x) \, d\mu_R(y)}{\vol(B_R^n)} = \frac{\int \widehat{f} \widehat{\chi}_R^2\, d \mu_P}{\int \widehat{\chi}_R^2\, d \mu_P}.
\]
We can already see a problem: we would like the mass of
$\widehat{\chi}_R^2$ to be concentrated near $0$ as $R \to \infty$.
However, $0$ is not even contained within the support of the Plancherel
measure $\mu_P$, so this cannot possibly work.  To see how badly it fails, we
return to radial functions on $\HH^n$ via
\[
\int \widehat{f} \widehat{\chi}_R^2 \, d \mu_P = \int_{\HH^n} f \cdot (\chi_R * \chi_R).
\]
(We use $\cdot$ for multiplication here to avoid confusion with $f$
applied to an argument.)  The function $(\chi_R
* \chi_R)/\vol(B_R^n)$ measures the fraction of overlap between two
balls of radius $R$ whose centers are a given distance apart.
In the limit as $R \to \infty$, this function does not converge
to $1$, as it does in Euclidean space.  Instead, when the
distance between the centers is $z$ it converges to
\begin{equation} \label{eq:asympoverlap}
\frac{B\left(\frac{1}{1+e^z}; \frac{n-1}{2},\frac{n-1}{2}\right)}{B\left(\frac{1}{2};\frac{n-1}{2},\frac{n-1}{2}\right)},
\end{equation}
where
\[
B(u;\alpha,\beta) = \int_0^u t^{\alpha-1} (1-t)^{\beta-1} \, dt
\]
is the incomplete beta function. (See
Appendix~\ref{app:overlap} for the calculation.)  Note that
\eqref{eq:asympoverlap} equals $1$ when $z=0$ and vanishes in
the limit as $z \to \infty$.

Thus,
\[
\frac{\iint f(x,y) \, d\mu_R(x) \, d\mu_R(y)}{\vol(B_R^n)}
\]
converges to the integral of $f$ times a function that takes values
between $0$ and $1$.  Variants of this approach, for example replacing
$\chi_R$ with a smoother function such as the heat kernel, fail for
essentially the same reason.  In Euclidean space the heat
kernel converges to the constant function $1$ as time tends to
infinity, if we rescale it so its value at the origin is fixed as $1$.
In other words, flowing heat becomes nearly uniformly distributed over
time.  However, in hyperbolic space that is not true (heat kernel
asymptotics can be found in \cite{DM88}).

\subsection{Periodic packings}

We can prove a variant of Conjecture~\ref{conjecture:hypLP} in the special
case of periodic packings, i.e., packings that are invariant under a
discrete, finite-covolume group of isometries.\footnote{Note that the
definition of ``periodic'' varies between papers: \cite{BR03} requires a
cocompact group, while \cite{B03} does not.}  This was first proved by Cohn,
Lurie, and Sarnak in unpublished work; here, we give a proof under weaker
hypotheses but using the same fundamental approach. It is not known in
general whether periodic packings come arbitrarily close to the optimal
Bowen-Radin packing density, although this has been proved for the hyperbolic
plane \cite{B03}.  Maximizing the density for a single-orbit packing in
$\HH^n$ under a cocompact group is equivalent to maximizing the systolic
ratio of an $n$-dimensional compact hyperbolic manifold (see \cite{K07} for
background on systolic geometry), which makes periodic packings a
particularly important case.

Given a periodic packing with balls of radius $r$, let $\Gamma
\subset G$ be its symmetry group. By assumption,
$\Gamma\backslash\HH^n$ has finite volume. Suppose the spheres
in the packing are centered on the orbits $\Gamma x_1, \dots,
\Gamma x_N$, and let $\Gamma_i$ be the stabilizer of $x_i$ in
$\Gamma$. Then the density of the packing is the fraction of a
fundamental domain covered by balls.  If the stabilizers are
trivial, then this fraction is simply $N
\vol(B_r^n)/\vol(\Gamma \backslash \HH^n)$.  More generally,
$\Gamma_i$ preserves the ball centered at $x_i$, and only a
$1/|\Gamma_i|$ fraction of this ball will lie in any given
fundamental domain (specifically, one element of each
$\Gamma_i$-orbit). Thus, the density is
\[
\frac{\vol(B_r^n)}{\vol(\Gamma \backslash \HH^n)} \sum_{i=1}^N \frac{1}{|\Gamma_i|}.
\]
It is the same as the density of the corresponding Bowen-Radin measure.\footnote{Proposition~1
in \cite{BR03} is stated for the cocompact case, but the proof works for the finite
covolume case as well.}

The proof of the linear programming bounds is based on the Selberg
trace formula \cite{S56}.  More precisely, we use a pre-trace formula
that plays the role of the Poisson summation formula in \cite{CE03}.
To minimize the background required, we derive it from the
spectral theory of the Laplacian on $\Gamma\backslash\HH^n$.  Note that
Selberg did not publish a complete proof of the trace formula.  For a
detailed proof, see \cite{F67} for $\HH^2$, \cite{V73} for hyperbolic
spaces under some mild hypotheses on the discrete group, \cite{EGM98}
for the three-dimensional case (using techniques that work in greater
generality), and \cite{CS80} for hyperbolic spaces in general.

Call a function $f \colon [0,\infty) \to \RR$ \emph{admissible} on $\HH^n$
if
\begin{enumerate}
\item it is continuous and integrable on $\HH^n$,

\item for every discrete subgroup $\Gamma \subset
G$ for which $\Gamma\backslash\HH^n$ has finite volume,
\[
\sum_{\gamma \in \Gamma} f(\gamma x, y)
\]
converges absolutely for all $x,y \in \HH^n$ and uniformly on compact
subsets of $\HH^n$, and

\item for each fixed $y$,
\[
x \mapsto \sum_{\gamma \in \Gamma} f(\gamma x, y)
\]
is in $L^2(\Gamma\backslash\HH^n).$
\end{enumerate}

All of the functions constructed in the proof of Theorem~\ref{theorem:hyperbolicRodemich}
are admissible by the following lemma.

\begin{lemma}
Every continuous, compactly supported function is admissible.
\end{lemma}

\begin{proof}
Let $f \colon [0,\infty) \to \RR$ be continuous, and suppose $f$ vanishes outside $[0,r]$.
In the sum
\[
\sum_{\gamma \in \Gamma} f(\gamma x, y),
\]
the term $f(\gamma x,y)$ is nonzero only if $d(\gamma x, y) \le r$.  Absolute
and uniform convergence on compact subsets is easy: if $x$ and $y$ are confined to
a compact set $K$, then there is a finite subset $S$ of $\Gamma$ such that $d(\gamma x, y) > r$
whenever $\gamma \not\in S$, because $\Gamma$ acts discontinuously on $\HH^n$ (see \S5.3 in \cite{R06}).
Then only the terms with $\gamma \in S$ contribute to the sum.

To complete the proof, we will show that the sum is bounded for fixed $y$,
based on the proof of Lemma~2.6.1 in \cite{EGM98}. Choose $\varepsilon>0$ so
that the balls $B_{\varepsilon}^n(\gamma^{-1} y)$ form a sphere packing; in
other words, the only intersections between them come from the stabilizer
$\Gamma_y$ of $y$. Then the number of $\gamma$ for which $d(\gamma x, y) \le
r$ is at most
\[
\frac{|\Gamma_y| \vol(B^n_{r+\varepsilon})}{\vol(B^n_{\varepsilon})},
\]
because at most $\vol(B^n_{r+\varepsilon})/\vol(B^n_{\varepsilon})$ of the balls
$B_{\varepsilon}^n(\gamma^{-1} y)$ can fit into $B_{r+\varepsilon}^n(x)$, and each occurs
for $|\Gamma_y|$ choices of $\gamma$.  This bound
is independent of $x$, and
\[
\sum_{\gamma \in \Gamma} |f(\gamma x, y)| \le \frac{|\Gamma_y| \vol(B^n_{r+\varepsilon})}{\vol(B^n_{\varepsilon})} \max_{[0,r]} |f|.
\]
\end{proof}

The following lemma provides more examples of admissible functions.
It is essentially Lemma~1.4 in \cite{V73}, where it is
attributed to Selberg, and we include the proof here for completeness.

\begin{lemma}[Selberg]
Suppose $f \colon [0,\infty) \to \RR$ is continuous and integrable on $\HH^n$, and
there exist constants $c_1,c_2$ such that for all $x,y \in \HH^n$,
\[
|f(x,y)| \le c_1 \int_{d(z,x) \le c_2} |f(y,z)| \, dz.
\]
Then $f$ is admissible on $\HH^n$.
\end{lemma}

\begin{proof}
Because $\Gamma$ acts discontinuously on $\HH^n$, each of the balls
$B_{c_2}^n(\gamma x)$ with $\gamma \in \Gamma$ intersects only a finite
number of these balls, say $N(x)$ of them (counting itself), and this
function $N$ is bounded on compact sets. Then
\begin{equation} \label{eq:integralupperbd}
\begin{split}
\sum_{\gamma \in \Gamma} |f(\gamma x, y)| &\le c_1 \sum_{\gamma \in \Gamma} \int_{B_{c_2}^n(\gamma x)} |f(y,z)| \, dz\\
& \le
c_1 N(x) \int_{\HH^n} |f(y,z)| \, dz.
\end{split}
\end{equation}
The left side is invariant under switching $x$ and $y$, while the right side is
independent of $y$, from which it follows
that
\[
\sum_{\gamma \in \Gamma} |f(\gamma x, y)|
\]
is bounded for each fixed $y$.  All that remains is to verify uniform
convergence on compact sets, which is not hard to check as follows.  Suppose
$x$ and $y$ are restricted to a compact set $K$, and let $r>0$.  In the upper
bound $\int_{\HH^n} |f(y,z)| \, dz$ from \eqref{eq:integralupperbd}, all $z
\in B^n_r(y)$ come from a finite subset $S$ of $\Gamma$ depending only on $K$
and $r$. It follows that
\[
\sum_{\gamma \not\in S} |f(\gamma x, y)| \le c_1 N(x) \int_{d(z,y) \ge r} |f(y,z)| \, dz,
\]
and the upper bound tends to zero as $r \to \infty$.
\end{proof}

The only place where we require the trace formula machinery is the
proof of the following lemma:

\begin{lemma} \label{lem:hypsumposdef}
Let $f \colon [0,\infty) \to \RR$ be admissible on $\HH^n$ and satisfy
$\widehat{f}(\lambda) \ge 0$ for all $\lambda\ge0$.  If $\Gamma$ is a
discrete subgroup in $G$ with finite covolume, then the function $F$
defined by
\[
F(x,y) = \sum_{\gamma \in \Gamma} f(\gamma x, y)
\]
is positive definite on $\HH^n \times \HH^n$.  Furthermore, $F -
\widehat{f}(0)/\vol(\Gamma \backslash \HH^n)$ remains positive
definite.
\end{lemma}

In other words, for all $x_1,\dots,x_N \in \HH^n$, the $N \times N$
matrix with entries $F(x_i,x_j) - \widehat{f}(0)/\vol(\Gamma \backslash
\HH^n)$ is positive semidefinite.

\begin{proof}
First, suppose $\Gamma\backslash \HH^n$ is compact, so the spectrum of
the Laplacian is discrete.  Let $v_0,v_1,\dots$ be the orthonormal
eigenfunctions of the Laplacian on $\Gamma\backslash \HH^n$, viewed as
periodic functions on $\HH^n$, and let $\lambda_0 \le\lambda_1 \le
\dots$ be the corresponding eigenvalues. The sum
\[
\sum_{\gamma \in \Gamma} f(\gamma x, y)
\]
is periodic modulo $\Gamma$ as a function of $x$ (or $y$), so we can
expand it in terms of the eigenfunctions of the Laplacian.  We have
\[
\sum_{\gamma \in \Gamma} f(\gamma x, y) \simeq \sum_{i=0}^\infty v_i(x) \int_{\Gamma \backslash \HH^n}
\left(\sum_{\gamma \in \Gamma} f(\gamma z, y)\right) \overline{v_i(z)} \, dz,
\]
where $\simeq$ denotes $L^2$ convergence.  The coefficients unfold to
\[
\int_{\HH^n} f(z,y) \overline{v_i(z)} \, dz,
\]
and we can rotationally symmetrize about $y$, which turns $v_i(z)$ into
$v_i(y) P_{\lambda_i}(z,y)$ and the coefficient into
\[
\overline{v_i(y)} \int_{\HH^n} f(z,y)  \overline{P_{\lambda_i}(z,y)} \, dz
= \overline{v_i(y)} \widehat{f}(\lambda_i).
\]
(The conjugate on $P_{\lambda_i}$ does not matter, because this
function is real-valued.)  Thus,
\[
\sum_{\gamma \in \Gamma} f(\gamma x, y) \simeq \sum_{i=0}^\infty \widehat{f}(\lambda_i)
v_i(x) \overline{v_i(y)}.
\]

The functions $(x,y) \mapsto v_i(x) \overline{v_i(y)}$ are clearly
positive definite, and the coefficients $\widehat{f}(\lambda_i)$ are
nonnegative.  Furthermore, positive definiteness is preserved under
pointwise convergence. However, this expansion may not converge
pointwise. Fortunately, $L^2$ convergence implies that a subsequence
converges pointwise almost everywhere (see, for example, Theorem~3.12
in \cite{R87}).  Thus, for almost all $x_1,\dots,x_N \in \HH^n$, the
matrix with entries $F(x_i,x_j)$ is positive semidefinite, and the same
holds for all $x_1,\dots,x_N$ by continuity.  Furthermore, $v_0$ is the
constant eigenfunction, so $v_0 \overline{v_0}$ must be $1/\vol(\Gamma
\backslash \HH^n)$ by orthonormality, and thus $F -
\widehat{f}(0)/\vol(\Gamma \backslash \HH^n)$ is also positive
definite.

All that remains is to deal with the case when $\Gamma \backslash
\HH^n$ has finite volume but is not compact.  Harmonic analysis on the
quotient is quite a bit more involved, because of continuous spectrum
coming from the cusps, but a completely analogous argument works.
Suppose $\Gamma\backslash \HH^n$ has $h$ cusps.  For $1 \le k \le h$
and $s \in \CC$ with $\Re(s) = (n-1)/2$, there is an Eisenstein series $x
\mapsto E_k(x,s)$, which is an eigenfunction with eigenvalue $s(n-1-s)$.
Note that these Eisenstein series are not in $L^2(\Gamma\backslash\HH^n)$.
When $s = (n-1)/2+it$, the eigenvalue becomes $(n-1)^2/4+t^2$, so it is
contained in the support $[(n-1)^2/4,\infty)$ of the Plancherel measure.
The spectral resolution is now
\[
\begin{split}
\sum_{\gamma \in \Gamma} f(\gamma x, y) \simeq &
\sum_{j=0}^\infty \widehat{f}(\lambda_j)
v_j(x) \overline{v_j(y)}\\
& \phantom{} + \frac{1}{4\pi} \sum_{k=1}^h \int_{-\infty}^\infty
\widehat{f}\left(\frac{(n-1)^2}{4}+t^2\right) \cdot {}\\ 
& \qquad \qquad \qquad \ E_k\left(x,\frac{n-1}{2}+it\right) \overline{E_k\left(y,\frac{n-1}{2}+it\right)} \, dt.
\end{split}
\]
See (7.30) in \cite[p.~75]{CS80} for the underlying decomposition of
$L^2(\Gamma\backslash\HH^n)$, although that formula is missing the factor of $1/(4\pi)$.
This expansion means that the left side is the $L^2$ limit of
\[
\begin{split}
&\sum_{j=0}^N \widehat{f}(\lambda_j)
v_j(x) \overline{v_j(y)}\\
&\phantom{} + \frac{1}{4\pi} \sum_{k=1}^h \int_{-T}^T
\widehat{f}\left(\frac{(n-1)^2}{4}+t^2\right) E_k\left(x,\frac{n-1}{2}+it\right) \overline{E_k\left(y,\frac{n-1}{2}+it\right)} \, dt
\end{split}
\]
as $N$ and $T$ tend to infinity.  Now the proof proceeds as in the compact case.
\end{proof}

The proof of Lemma~\ref{lem:hypsumposdef} depends on the hypothesis that
$\widehat{f} \ge 0$ everywhere, not just on the support $[(n-1)^2/4,\infty)$
of the Plancherel measure, because the eigenvalues of the Laplacian on
$\Gamma \backslash \HH^n$ need not be contained in the support: $0$ is always
an eigenvalue, and there can be others between $0$ and $(n-1)^2/4$. Selberg's
eigenvalue conjecture \cite{S65,S95} says there are no nonzero eigenvalues in
$(0,1/4)$ for congruence subgroups of $\SL_2(\ZZ)$, but that is a special
arithmetic property not shared by more general groups.

\begin{theorem}[Cohn, Lurie, and Sarnak] \label{thm:hypLP}
Let $f \colon [0,\infty) \to \RR$ be admissible on $\HH^n$, and suppose
$f(x) \le 0$ for all $x \ge 2r$ while $\widehat{f}(\lambda) \ge0$ for
all $\lambda>0$ and $\widehat{f}(0)>0$. Then every periodic packing in
$\HH^n$ using balls of radius $r$ has density at most
\[
\vol(B_r^n) \frac{f(0)}{\widehat{f}(0)}.
\]
\end{theorem}

\begin{proof}
Consider a periodic packing consisting of orbits $\Gamma x_1, \dots,
\Gamma x_N$ of a finite-covolume group $\Gamma$, and let $\Gamma_i$ be
the stabilizer of $x_i$ in $\Gamma$.  Recall that the density of the
packing is
\[
\frac{\vol(B_r^n)}{\vol(\Gamma \backslash \HH^n)} \sum_{i=1}^N \frac{1}{|\Gamma_i|}.
\]

Let
\[
F(x,y) = \sum_{\gamma \in \Gamma} f(\gamma x,y).
\]
By Lemma~\ref{lem:hypsumposdef},
\[
\sum_{i,j=1}^N \frac{1}{|\Gamma_i|\,|\Gamma_j|} 
\left(F(x_i,x_j) -\frac{\widehat{f}(0)}{\vol(\Gamma\backslash\HH^n)} \right) \ge 0,
\]
which amounts to
\[
\sum_{i,j=1}^N \frac{F(x_i,x_j)}{|\Gamma_i|\,|\Gamma_j|} 
\ge \left(\sum_{i=1}^N \frac{1}{|\Gamma_i|}\right)^2
\frac{\widehat{f}(0)}{\vol(\Gamma\backslash\HH^n)}.
\]
On the other hand, $F(x_i,x_j) \le 0$ for $i \ne j$, because all the
terms in the sum defining $F$ are nonpositive in that case, and
$F(x_i,x_i) \le |\Gamma_i| f(0)$, because there are $|\Gamma_i|$ group
elements $\gamma$ for which $\gamma x_i = x_i$.  Thus,
\[
\sum_{i=1}^N  \frac{|\Gamma_i|f(0)}{|\Gamma_i|^2} \ge
\left(\sum_{i=1}^N \frac{1}{|\Gamma_i|}\right)^2
\frac{\widehat{f}(0)}{\vol(\Gamma\backslash\HH^n)}.
\]
We conclude that
\[
\frac{\vol(B_r^n)}{\vol(\Gamma \backslash \HH^n)} \sum_{i=1}^N \frac{1}{|\Gamma_i|}
\le \vol(B_r^n) \frac{f(0)}{\widehat{f}(0)},
\]
as desired.
\end{proof}

\section*{Acknowledgments}

We thank the referees for detailed and valuable feedback on our manuscript,
and we thank Lewis Bowen, Donald Cohn, Thomas Hales, Jacob Lurie, and Peter
Sarnak for helpful conversations.

\appendix

\section{Numerical computation of Euclidean density bounds}
\label{app:numerical}

Before the Cohn-Elkies paper \cite{CE03}, the three best upper bounds known
for sphere packing in $\RR^n$ with $n>3$ were those of Rogers \cite{Rog58},
Levenshtein \cite{Lev79}, and Kabatiansky and Levenshtein \cite{KL78}.  The
Rogers bound was the best known for $4 \le n \le 95$, the Levenshtein bound
for $96 \le n \le 114$, and the Kabatiansky-Levenshtein bound for $n \ge
115$.  See Table~\ref{table:numbers} for numerical data. Note that the
asymptotic decay rates are not apparent from the behavior in low dimensions.

Table~\ref{table:numbers} differs from the bounds presented in Table~1.3 of
\cite{CS99}.  Specifically, \cite[p.~20]{CS99} says that the
Kabatiansky-Levenshtein bound improves on the Rogers bound for $n \ge 43$,
and Table~1.3 lists some special cases, but our computations of the
Kabatiansky-Levenshtein bound disagree. To help resolve this discrepancy, we
will specify how we computed all these bounds.

The Rogers bound is conceptually simple: it is the fraction of a regular
simplex covered by congruent balls centered at its vertices and tangent to each other.
However, it is somewhat complicated to compute explicitly. Based on Chapter~7
of \cite{Zon99}, we used the formula
\[
\frac{(n+1)!}{(n/2)!} \cdot \frac{\pi^{(n-1)/2}}{ 2^{3n/2}}
\int_{-\infty}^\infty e^{(n+1)\big(n/2 - \sqrt{2n} u i - u^2\big)}
\left(1- \erf\left(\sqrt{n/2}-ui\right)\right)^n du
\]
for the Rogers bound in $\RR^n$, where $\erf$ denotes the error function
\[
\erf(x) = \frac{2}{\sqrt{\pi}}\int_0^x e^{-t^2} \, dt.
\]

The Levenshtein bound in $\RR^n$ equals
\[
\frac{j_{n/2}^n}{(n/2)!^2 4^n},
\]
where $j_{n/2}$ is the first positive root of the Bessel function $J_{n/2}$.

For the Kabatiansky-Levenshtein bound, let $t_{n,k}$ denote the largest root of
the Gegenbauer polynomial $C^{n/2-1}_k$ of degree $k$.  Kabatiansky and Levenshtein
proved that
\[
A(n,\theta) \le \frac{4 \binom{k+n-2}{k}}{1-t_{n,k+1}}
\]
whenever $\cos \theta \le t_{n,k}$.  Combining this bound for $A(n+1,\theta)$
with \eqref{eq:lift} and taking $\cos \theta = t_{n+1,k}$ to minimize $\sin(\theta/2)$,
we obtain a sphere packing density bound of
\[
\inf_k \left(\frac{1-t_{n+1,k}}{2}\right)^{n/2}\frac{4 \binom{k+n-1}{k}}{1-t_{n+1,k+1}}
\]
in $\RR^n$.  We have not rigorously analyzed how this bound depends on $k$,
but the infimum appears to be achieved, in fact at a unique local minimum.
In our numerical calculations, we search consecutively through $k=1,2,\dots$
until we find the first local minimum.

Our new bound in this paper (Proposition~\ref{prop:compare}, applied using
the Kabatiansky-Levenshtein bound on $A(n,\theta)$) is given by
\[
\min_k \left(\frac{1-t_{n,k}}{2}\right)^{n/2}\frac{4 \binom{k+n-2}{k}}{1-t_{n,k+1}},
\]
where the minimum is over $k$ satisfying $t_{n,k} \le 1/2$ (which corresponds
to $\theta \ge \pi/3$).  Note that $t_{n,k}$ is an increasing function of $k$
and $t_{n,k} \to 1$ as $k \to \infty$.

\begin{table}
\caption{Upper bounds for sphere packing density in $\RR^n$. The last column
gives the new bound from Proposition~\ref{prop:compare}, applied using the
Kabatiansky-Levenshtein bound for $A(n,\theta)$.  All numbers are rounded
up.} \label{table:numbers}
\begin{center}
\begin{tabular}{ccccc}
\toprule
$n$ & Rogers & Levenshtein & K.-L. & Prop.~\ref{prop:compare} \\
\midrule
$12$ & $8.759 \times 10^{-2\phantom{0}}$  & $1.065 \times 10^{-1}\phantom{0}$  & $1.038 \times 10^{0\phantom{-00}}$  & $9.666 \times 10^{-1\phantom{00}}$\\
$24$ & $2.456 \times 10^{-3\phantom{0}}$  & $3.420 \times 10^{-3}\phantom{0}$  & $2.930 \times 10^{-2\phantom{00}}$  & $2.637 \times 10^{-2\phantom{00}}$\\
$36$ & $5.527 \times 10^{-5\phantom{0}}$  & $8.109 \times 10^{-5}\phantom{0}$  & $5.547 \times 10^{-4\phantom{00}}$  & $4.951 \times 10^{-4\phantom{00}}$\\
$48$ & $1.128 \times 10^{-6\phantom{0}}$  & $1.643 \times 10^{-6}\phantom{0}$  & $8.745 \times 10^{-6\phantom{00}}$  & $7.649 \times 10^{-6\phantom{00}}$\\
$60$ & $2.173 \times 10^{-8\phantom{0}}$  & $3.009 \times 10^{-8\phantom{0}}$  & $1.223 \times 10^{-7\phantom{00}}$  & $1.046 \times 10^{-7\phantom{00}}$\\
$72$ & $4.039 \times 10^{-10}$ & $5.135 \times 10^{-10}$ & $1.550 \times 10^{-9\phantom{00}}$  & $1.322 \times 10^{-9\phantom{00}}$\\
$84$ & $7.315 \times 10^{-12}$ & $8.312 \times 10^{-12}$ & $1.850 \times 10^{-11\phantom{0}}$ & $1.574 \times 10^{-11\phantom{0}}$\\
$96$ & $1.300 \times 10^{-13}$ & $1.291 \times 10^{-13}$ & $2.111 \times 10^{-13\phantom{0}}$ & $1.786 \times 10^{-13\phantom{0}}$\\
$108$& $2.277 \times 10^{-15}$ & $1.937 \times 10^{-15}$ & $2.320 \times 10^{-15\phantom{0}}$ & $1.942 \times 10^{-15\phantom{0}}$\\
$120$& $3.940 \times 10^{-17}$ & $2.826 \times 10^{-17}$ & $2.452 \times 10^{-17\phantom{0}}$ & $2.051 \times 10^{-17\phantom{0}}$\\
$240$ & $6.739 \times 10^{-35}$ & $4.888 \times 10^{-36}$ & $1.542 \times 10^{-37\phantom{0}}$ & $1.267 \times 10^{-37\phantom{0}}$\\
$360$ & $8.726 \times 10^{-53}$ & $3.522 \times 10^{-55}$ & $3.689 \times 10^{-58\phantom{0}}$ & $3.003 \times 10^{-58\phantom{0}}$\\
$480$ & $1.007 \times 10^{-70}$ & $1.643 \times 10^{-74}$ & $5.536 \times 10^{-79\phantom{0}}$ & $4.484 \times 10^{-79\phantom{0}}$\\
$600$ & $1.090 \times 10^{-88}$ & $5.847 \times 10^{-94}$ & $6.233 \times 10^{-100}$ & $5.036 \times 10^{-100}$\\
\bottomrule
\end{tabular}
\end{center}
\end{table}

Using the above formulas, we have computed these bounds for $2 \le n \le 128$
and several larger values of $n$ using \textsc{Mathematica} 9.0.1 \cite{W13}, to
obtain the cross-over points listed above and the data in
Table~\ref{table:numbers}. Strictly speaking, our calculations are not
rigorous, because we have not proved bounds for floating-point error.
Furthermore, we have not proved that the bounds never cross again.  There is
no theoretical reason why these issues could not be addressed, but it
would take some work.

As we mentioned above, our calculations disagree with those in \cite{CS99}.
For example, Table~1.3 of \cite{CS99} says the Kabatiansky-Levenshtein bound
for $\RR^{48}$ is
\[
2^{15.27} \cdot \frac{\pi^{24}}{24!} \approx 5.44 \times 10^{-8},
\]
which is substantially less than the $8.745 \times 10^{-6}$ listed in
Table~\ref{table:numbers}. Page~265 of \cite{CS99} explains that the
Kabatiansky-Levenshtein calculations in Table~1.3 were carried out using
information about the Gegenbauer polynomial roots from \cite[p.~12]{KL78}.
The results in \cite[p.~12]{KL78} are asymptotic formulas that are not
accurate in low dimensions, and we hypothesize that this explains the
discrepancy, although we do not know how to obtain the numbers quoted in
Table~1.3 of \cite{CS99}.

For comparison, the Cohn-Elkies linear programming bound improves on all
these bounds for $4 \le n \le 128$.  Improving on the Kabatiansky-Levenshtein
bound is no surprise by Theorem~\ref{thm:LP-match}, and Proposition~6.1 of
\cite{CE03} says the linear programming bound is always at least as strong as
the Levenshtein bound. Improving on the Rogers bound is the only part we
cannot explain conceptually, and it can be verified using an auxiliary
function with eight forced double roots in the numerical technique from
Section~7 of \cite{CE03}. We do not report linear programming bounds in
Table~\ref{table:numbers}, because we have not completed enough calculations
to give truly representative data.  As the dimension grows, the number of
forced double roots required to optimize the bound grows as well.  Using only
eight of them substantially weakens the bound, but it already suffices to improve on
the other bounds when $4 \le n \le 128$.  For example, using eight forced double
roots leads to a bound of $1.164 \times 10^{-17}$ when $n=120$.

\section{Overlap of balls in hyperbolic space}
\label{app:overlap}

\begin{proposition} \label{prop:overlap}
If $n \ge 2$ and $x_1$ and $x_2$ are points in $\HH^n$ at distance $r$ from
each other, then
\[
\lim_{R \to \infty}\frac{\vol\big(B_R^n(x_1) \cap B_R^n(x_2)\big)}{\vol(B_R^n)} =
\frac{B\left(\frac{1}{1+e^r};
\frac{n-1}{2},\frac{n-1}{2}\right)}{B\left(\frac{1}{2};\frac{n-1}{2},\frac{n-1}{2}\right)}.
\]
\end{proposition}

Recall that
\[
B(u;\alpha,\beta) = \int_0^u t^{\alpha-1}
(1-t)^{\beta-1} \, dt
\]
is the incomplete beta function.  To prove
Proposition~\ref{prop:overlap}, we will compute the convolution
$(\chi_R * \chi_R)(r)/\vol(B_R^n)$ on $\HH^n$, where $\chi_R$ is
the characteristic function of a ball of radius $R$, and take
the limit as $R \to \infty$.

First, we observe that given two radial functions $f_1$ and
$f_2$ on $\HH^n$, their convolution can be computed as follows.
Let $x_1$ and $x_2$ be two points in $\HH^n$ at distance $r$
from each other, and consider a third point $z$ at distances
$r_1$ and $r_2$ from $x_1$ and $x_2$, respectively.  Then
$(f_1*f_2)(r)$ is the integral of $f_1(r_1) f_2(r_2)$ over all
$z \in \HH^n$.  To write it down explicitly, we can use polar
coordinates centered at $x_1$.  Let $u = \cos \angle x_2 x_1
z$.  Then
\[
(f_1 * f_2)(r) = \frac{2\pi^{(n-1)/2}}{\Gamma((n-1)/2)}
\int_0^\infty \int_{-1}^1 f_1(r_1) f_2(r_2)
\big(\sinh^{n-1} r_1\big) (1-u^2)^{(n-3)/2} \, du \, dr_1,
\]
where we view $r_2$ as a function of $r_1$ and $u$ via the
hyperbolic law of cosines. Changing variables from $u$ to $r_2$
yields
\begin{equation} \label{eq:conv}
\frac{2\pi^{(n-1)/2}}{\Gamma((n-1)/2)}
\iint f_1(r_1) f_2(r_2)
\frac{\sinh r_1 \sinh r_2}{\sinh^{n-2} r} C(r,r_1,r_2)^{(n-3)/2} \, dr_1 \, dr_2,
\end{equation}
where the integral is over all $r_1$ and $r_2$ such that
$r,r_1,r_2$ form the side lengths of a triangle, and
\[
C(r,r_1,r_2) = 1 - \cosh^2 r - \cosh^2 r_1 - \cosh^2 r_2 + 2 \cosh r \cosh r_1 \cosh r_2.
\]

Now we apply \eqref{eq:conv} to $f_1=f_2=\chi_R$ and divide by
$\vol(B_R^n)$, which is asymptotic to
$2\pi^{n/2}e^{(n-1)R}/((n-1)2^{n-1}\Gamma(n/2))$ as $R \to
\infty$. We find that $(\chi_R * \chi_R)(r)/\vol(B_R^n)$ is
asymptotic to
\begin{equation} \label{eq:asympconv}
\frac{(n-1) 2^{n-1}}{B((n-1)/2,1/2)} e^{-(n-1)R}
\iint_X \frac{\sinh r_1 \sinh r_2}{\sinh^{n-2} r} C(r,r_1,r_2)^{(n-3)/2} \, dr_1 \, dr_2,
\end{equation}
where $B(\alpha,\beta) = B(1;\alpha,\beta) =
\Gamma(\alpha)\Gamma(\beta)/\Gamma(\alpha+\beta)$ denotes the
beta function and $X$ is the set of $(r_1,r_2) \in [0,R]^2$ for
which $r,r_1,r_2$ form a triangle.

We now change variables in \eqref{eq:asympconv} from $r_1$ and
$r_2$ to $x = r_1-r_2$ and $y = 2R-r_1-r_2$.  In the new
variables, the domain $X$ of integration becomes
\[
\{(x,y) : |x| \le r \textup{ and } |x| \le y \le 2R-r\}.
\]
Expanding the hyperbolic trigonometric functions in terms of exponentials
shows that
\[
\sinh r_1 =
\frac{e^{R+(x-y)/2}}{2} + O(1),
\]
\[
\sinh r_2 = \frac{e^{R-(x+y)/2}}{2} + O(1),
\]
and
\[
C(r,r_1,r_2) = e^{2R-y} \frac{\cosh r - \cosh x}{2} + O(1),
\]
where the $O(1)$ terms depend on $r$ but not $x$, $y$, or $R$.  By the mean value
theorem,
\[
C(r,r_1,r_2)^{(n-3)/2} = \left( e^{2R-y} \frac{\cosh r - \cosh x}{2} \right)^{(n-3)/2}
+ O\left(e^{(n-5)R-(n-5)y/2}\right),
\]
where the constant in the big $O$ term depends only on $n$ and $r$.
Now as $R \to \infty$, the integral \eqref{eq:asympconv} converges to
\[
\frac{(n-1) 2^{(n-5)/2}}{B((n-1)/2,1/2)\sinh^{n-2} r}
\int_{-r}^r \int_{|x|}^\infty e^{-(n-1)y/2} (\cosh r - \cosh x)^{(n-3)/2} \, dy \, dx.
\]
We can evaluate the $y$ integral explicitly, and the remaining integrand
is an even function of $x$. The integral thus equals
\[
\frac{2^{(n-1)/2}}{B((n-1)/2,1/2)\sinh^{n-2} r}
\int_{0}^r  e^{-(n-1)x/2} (\cosh r - \cosh x)^{(n-3)/2} \, dx.
\]
Finally, we change to a new variable
\[
t = \frac{e^{-x}-e^{-r}}{e^r - e^{-r}}
\]
to arrive at
\[
\frac{2^{n-1}}{B((n-1)/2,1/2)} \int_0^{1/(1+e^r)} \big(t(1-t)\big)^{(n-3)/2} \, dt.
\]
It follows from the duplication formula for the gamma function
that this expression equals
\[
\frac{\int_0^{1/(1+e^r)} \big(t(1-t)\big)^{(n-3)/2} \, dt}{B((n-1)/2,(n-1)/2)/2}.
\]
This completes the proof of Proposition~\ref{prop:overlap}.


\begin{thebibliography}{MRRW77}

\bibitem[BM07]{BM07}  A.~Barg and O.~R.~Musin, \emph{Codes in spherical
    caps}, Adv.\ Math.\ Commun.\ \textbf{1} (2007), 131--149.
    \MR{2262773} \doi{10.3934/amc.2007.1.131} \arXiv{math/0606734}

\bibitem[B65]{Bas65} L.~A. Bassalygo, \emph{New upper bounds for
    error-correcting codes} (Russian), Problemy
  Pereda\v ci Informacii \textbf{1} (1965), 41--44; English translation in
  Problems of Information Transmission \textbf{1} (1965), 32--35. \MR{0189920}

\bibitem[B33]{Boc33} S.~Bochner, \emph{Monotone Funktionen, Stieltjessche Integrale und harmonische Analyse},
    Math.\ Ann.\ \textbf{108} (1933), 378--410. \MR{1512856} \doi{10.1007/BF01452844}

\bibitem[B03]{B03} L.~Bowen, \emph{Periodicity and circle packings of the hyperbolic plane},
    Geom.\ Dedicata \textbf{102} (2003), 213--236. \MR{2026846} \doi{10.1023/B:GEOM.0000006580.47816.e9} \arXiv{math/0304344}

\bibitem[BR03]{BR03} L.~Bowen and C.~Radin, \emph{Densest packing of
    equal spheres in hyperbolic space}, Discrete Comput.\ Geom.\ \textbf{29}
    (2003), 23--39.  \MR{1946792} \doi{10.1007/s00454-002-2791-7}

\bibitem[BR04]{BR04} L.~Bowen and C.~Radin, \emph{Optimally dense
    packings of hyperbolic space}, Geometriae Dedicata \textbf{104} (2004), 37--59.
    \MR{2043953} \doi{10.1023/B:GEOM.0000022857.62695.15} \arXiv{math/0211417}

\bibitem[B78]{B78} K.~B\"or\"oczky, \emph{Packing of spheres in spaces
    of
    constant curvature}, Acta Math.\ Acad.\ Sci.\ Hungar.\ \textbf{32} (1978), 243--261.
    \MR{0512399} \doi{10.1007/BF01902361}

\bibitem[CS80]{CS80} P.~Cohen and P.~Sarnak, \emph{Selberg trace
    formula}, typed notes, 1980.  Chapters~6 and~7 available at
    \url{http://publications.ias.edu/sarnak/paper/496} and
    \url{http://publications.ias.edu/sarnak/paper/495}, respectively.

\bibitem[C10]{C10} H.~Cohn, \emph{Order and disorder in energy
    minimization}, Proceedings of the International Congress of Mathematicians.
    Volume IV, 2416--2443, Hindustan Book Agency, New Delhi, 2010.
    \MR{2827978} \doi{10.1142/9789814324359_0152} \arXiv{1003.3053}

\bibitem[CE03]{CE03} H.~Cohn and N.~Elkies, \emph{New upper bounds on
    sphere packings I}, Ann.\ of Math.\ (2) \textbf{157} (2003), 689--714. \MR{1973059}
    \doi{10.4007/annals.2003.157.689} \arXiv{math/0110009}

\bibitem[CK04]{CK04} H.~Cohn and A.~Kumar, \emph{The densest
    lattice in twenty-four dimensions}, Electron.\
    Res.\ Announc.\ Amer.\ Math.\ Soc.\ \textbf{10} (2004),
    58--67.  \MR{2075897} \doi{10.1090/S1079-6762-04-00130-1} \arXiv{math/0408174}

\bibitem[CK07]{CK07} H.~Cohn and A.~Kumar, \emph{Universally optimal
    distribution of points on
  spheres}, J.\ Amer.\ Math.\ Soc.\ \textbf{20} (2007), 99--148.
  \MR{2257398} \doi{10.1090/S0894-0347-06-00546-7} \arXiv{math/0607446}

\bibitem[CK09]{CK09} H.~Cohn and A.~Kumar, \emph{Optimality and
    uniqueness of the Leech lattice among lattices},
    Ann.\ of Math.\ (2) \textbf{170} (2009), 1003--1050. \MR{2600869} \doi{10.4007/annals.2009.170.1003}
    \arXiv{math/0403263}

\bibitem[CS99]{CS99} J.~H.~Conway and N.~J.~A.~Sloane, \emph{Sphere
    packings, lattices and groups}, third edition, Grundlehren der Mathematischen
    Wissenschaften \textbf{290}, Springer-Verlag, New York, 1999. \MR{1662447}

\bibitem[DM88]{DM88}  E.~B.~Davies and N.~Mandouvalos, \emph{Heat
    kernel
    bounds on hyperbolic space and Kleinian groups}, Proc.\ London Math.\
    Soc.\ (3) \textbf{57} (1988), 182--208. \MR{0940434} \doi{10.1112/plms/s3-57.1.182}

\bibitem[D72]{Del72} P.~Delsarte, \emph{Bounds for unrestricted
    codes, by linear programming},
  Philips Res.\ Rep.\ \textbf{27} (1972), 272--289. \MR{0314545}

\bibitem[D94]{D94} P.~Delsarte, \emph{Application and generalization of
    the MacWilliams transform in coding theory}, Proceedings of the
    15th Symposium on Information Theory in the Benelux (1994), 9--44.

\bibitem[DGS77]{DGS77} P.~Delsarte, J.~M.~Goethals, and J.~J.~Seidel,
    \emph{Spherical codes and designs}, Geometriae Dedicata \textbf{6} (1977),
    363--388. \MR{0485471} \doi{10.1007/BF03187604}

\bibitem[EGM98]{EGM98} J.~Elstrodt, F.~Grunewald, and J.~Mennicke,
    \emph{Groups acting on hyperbolic space: Harmonic
    analysis and number theory}, Springer-Verlag, Berlin, 1998.
    \MR{1483315}

\bibitem[F67]{F67} L.~D.~Faddeev, \emph{The eigenfunction expansion of
    the
    Laplace operator on the fundamental domain of a discrete group on
    the Loba\v{c}evski\u{\i} plane} (Russian), Trudy Moskov.\ Mat.\
    Ob\v{s}\v{c}.\ \textbf{17} (1967), 323--350; English translation in
    \emph{Transactions of the Moscow Mathematical Society for the year 1967} (Volume 17), 
    American Mathematical Society, 1969, pp.~356--386.  \MR{0236768}

\bibitem[G63]{G63} H.~Groemer, \emph{Existenzs\"atze f\"ur Lagerungen im Euklidischen Raum},
    Math.\ Z.\ \textbf{81} (1963), 260--278. \MR{0163222} \doi{10.1007/BF01111546}

\bibitem[H05]{Hal05} T.~C. Hales, \emph{A proof of the Kepler
    conjecture}, Ann.\ of Math.\ (2) \textbf{162} (2005), 1065--1185. \MR{2179728}
    \doi{10.4007/annals.2005.162.1065}

\bibitem[H08]{H08}  S.~Helgason, \emph{Geometric analysis on symmetric
    spaces}, second edition, Mathematical Surveys and Monographs
    \textbf{39}, American Mathematical Society, 2008.  \MR{2463854}

\bibitem[HST10]{HST10} A.~B.~Hopkins, F.~H.~Stillinger, and
    S.~Torquato, \emph{Spherical codes, maximal local packing density, and the golden ratio},
    J.\ Math.\ Phys.\ \textbf{51} (2010), 043302:1--6. \MR{2662486} \doi{10.1063/1.3372627} \arXiv{1003.3604}

\bibitem[HJ13]{HJ} R.~A.~Horn and C.~R.~Johnson, \emph{Matrix
    analysis}, second edition, Cambridge University Press, 2013.
    \MR{2978290}

\bibitem[KL78]{KL78} G.~A.~Kabatyanskii 
    and V.~I.~Levenshtein, \emph{Bounds for packings on a sphere and in space} (Russian),
    Problemy Pereda\v ci Informacii \textbf{14} (1978), 3--25; English translation in
    Problems of Information Transmission \textbf{14} (1978), 1--17. \MR{0514023}

\bibitem[K07]{K07} M.~G.~Katz, \emph{Systolic geometry and topology},
    Mathematical Surveys and Monographs \textbf{137}, American
    Mathematical Society, 2007. \MR{2292367}

\bibitem[LOV12]{LOV12} D.~de Laat, F.~M.~de Oliveira Filho,
    and F.~Vallentin, \emph{Upper bounds for packings of spheres of several radii},
    preprint, 2012, \arXiv{1206.2608}. 

\bibitem[L75]{L75} V.~I.~Levenshtein, \emph{Maximal packing density of
    $n$-dimensional Euclidean space with equal balls} (Russian), Mat.\ Zametki \textbf{18} (1975),
    301--311; English translation in Math.\ Notes \textbf{18} (1975), 765--771. \MR{0397565}

\bibitem[L79]{Lev79} V.~I.~Leven{\v{s}}te{\u\i}n, \emph{On bounds
    for packings in $n$-dimensional Euclidean space} (Russian), Dokl.\ Akad.\ Nauk SSSR \textbf{245} (1979),
    1299--1303; English translation in Soviet Math.\ Dokl.\ \textbf{20} (1979), 417--421. \MR{529659}

\bibitem[M99]{M99} T.~H.~Marshall, \emph{Asymptotic volume formulae and
    hyperbolic ball packing}, Ann.\ Acad.\ Sci.\ Fenn.\ Math.\ \textbf{24}
    (1999), 31--43.  \MR{1677993}

\bibitem[MRRW77]{MRRW} R.~J.~McEliece, E.~R.~Rodemich, H.~Rumsey, Jr.,
    and L.~R.~Welch, \emph{New
  upper bounds on the rate of a code via the Delsarte-MacWilliams
  inequalities}, IEEE Trans.\ Information Theory \textbf{IT-23} (1977),
  157--166. \MR{0439403} \doi{10.1109/TIT.1977.1055688}

\bibitem[N65]{N65} L.~Nachbin, \emph{The Haar integral},
    D.\ Van Nostrand Company, Inc., Princeton, NJ, 1965.  \MR{0175995}

\bibitem[R04]{R04} C.~Radin, \emph{Orbits of orbs: sphere packing meets
    Penrose tilings}, Amer.\ Math.\ Monthly \textbf{111} (2004), 137--149.
    \MR{2042761} \doi{10.2307/4145214}

\bibitem[R06]{R06} J.~G.~Ratcliffe, \emph{Foundations of hyperbolic manifolds},
    second edition, Graduate Texts in Mathematics \textbf{149}, Springer, New York, 2006.
    \MR{2249478}

\bibitem[R80]{R80} E.~R.~Rodemich, \emph{An inequality in coding
    theory}, Abstracts of Papers Presented to the American Mathematical
    Society \textbf{1} (1980), Abstract 773-05-39, page 15.

\bibitem[R58]{Rog58} C.~A.~Rogers, \emph{The packing of equal
    spheres}, Proc.\ London Math.\ Soc.\ (3)
  \textbf{8} (1958), 609--620. \MR{0102052} \doi{10.1112/plms/s3-8.4.609}

\bibitem[R87]{R87} W.~Rudin, \emph{Real and complex analysis}, third
    edition, McGraw-Hill Book Co., 1987.  \MR{0924157}

\bibitem[S95]{S95} P.~Sarnak, \emph{Selberg's eigenvalue conjecture},
    Notices Amer.\ Math.\ Soc.\ \textbf{42} (1995), 1272--1277.
    \MR{1355461}

\bibitem[S42]{Sch42} I.~J.~Schoenberg, \emph{Positive definite
    functions on spheres}, Duke Math.\ J.\ \textbf{9} (1942), 96--108. \MR{0005922} \doi{10.1215/S0012-7094-42-00908-6}

\bibitem[S56]{S56} A.~Selberg, \emph{Harmonic analysis and
    discontinuous
    groups in weakly symmetric Riemannian spaces with applications to
    Dirichlet series}, J.\ Indian Math.\ Soc.\ (N.S.) \textbf{20} (1956), 47--87.
    \MR{0088511}

\bibitem[S65]{S65} A.~Selberg, \emph{On the estimation of Fourier
    coefficients of modular forms}, Proc.\ Sympos.\ Pure Math.\ \textbf{8},
    pp.~1--15, American Mathematical Society, 1965. \MR{0182610}

\bibitem[T63]{T63} R.~Takahashi, \emph{Sur les repr\'esentations
    unitaires des groupes de Lorentz g\'en\'eralis\'es}, Bull.\ Soc.\
    Math.\ France \textbf{91} (1963), 289--433. \MR{0179296}

\bibitem[T82]{T82} A.~Terras, \emph{Noneuclidean harmonic analysis},
    SIAM Rev.\ \textbf{24} (1982), 159--193. \MR{0652465}
    \doi{10.1137/1024040}

\bibitem[V98]{V98} A.~Valette, \emph{On Godement's characterisation of
    amenability}, Bull.\ Austral.\ Math.\ Soc.\ \textbf{57} (1998), 153--158.
    \MR{1623844} \doi{10.1017/S0004972700031506}

\bibitem[V11]{V11} S.~Vance, \emph{Improved sphere packing lower bounds from Hurwitz lattices},
    Adv.\ Math.\ \textbf{227} (2011), 2144--2156. \MR{2803798} \doi{10.1016/j.aim.2011.04.016} \arXiv{1105.3779}

\bibitem[V13]{V13} A.~Venkatesh, \emph{A note on sphere packings in high dimension},
    Int.\ Math.\ Res.\ Not.\ IMRN \textbf{2013} (2013), 1628--1642. \MR{3044452} \doi{10.1093/imrn/rns096}

\bibitem[V73]{V73} A.~B.~Venkov, \emph{Expansion in automorphic
    eigenfunctions of the Laplace-Beltrami operator in classical
    symmetric spaces of rank one, and the Selberg trace formula} (Russian),
    Trudy Mat.\ Inst.\ Steklov \textbf{125} (1973), 6--55; English
    translation in Proc. Steklov Inst.\ Math.\ \textbf{125} (1973), 1--48. \MR{0562509}

\bibitem[W07]{W07} J.~A.~Wolf, \emph{Harmonic analysis on commutative
    spaces}, Mathematical Surveys and Monographs \textbf{142}, American
    Mathematical Society, 2007. \MR{2328043}

\bibitem[W13]{W13} Wolfram Research, Inc., \textsc{Mathematica}, Version 9.0.1,
Champaign, IL, 2013.

\bibitem[Z99]{Zon99} C.~Zong, \emph{Sphere packings}, Universitext,
    Springer-Verlag, New York, 1999. \MR{1707318}

\end{thebibliography}
\end{document}